\documentclass[10pt,fleqn]{article}
\usepackage[margin=0.8in]{geometry}
\usepackage{amsmath}
\usepackage{amssymb, amsthm, bm, upgreek}
\usepackage[hidelinks,bookmarksnumbered=true]{hyperref}
\usepackage{booktabs}
\usepackage{algorithm}
\usepackage{algpseudocode}
\usepackage{graphicx}
\usepackage{subcaption}
\usepackage{appendix}
\usepackage[numbers,sort&compress]{natbib}
\usepackage{IEEEtrantools}
\usepackage{authblk}

\newtheorem{theorem}{Theorem}
\newtheorem{prop}[theorem]{Proposition}
\newtheorem{assumption}{Assumption}
\newtheorem{remark}{Remark}

\newcommand{\vecf}{\bm}
\newcommand{\vect}{\mathbf}
\newcommand{\vectg}{\bm}
\newcommand{\tensor}[1]{\bm{\mathsf{#1}}}

\begin{document}
\bstctlcite{IEEEExample:BSTcontrol} 
\title{Matrix-Free Moment-Matching Method for Reduced-Order Modeling of  Quadratic-Bilinear Descriptor Systems with Multiple Inputs}
\author[1]{Zeyuan Gao}
\author[1]{Cheol W. Lee\thanks{Corresponding author:
\href{mailto:cheol@umich.edu}{cheol@umich.edu}}}
\author[2]{Oleg Zikanov}

\affil[1]{\small Department of Industrial and Manufacturing Systems Engineering,
University of Michigan-Dearborn, Dearborn, MI 48128, USA}
\affil[2]{Department of Mechanical Engineering,
University of Michigan-Dearborn, Dearborn, MI 48128, USA}
\date{}
\maketitle
\begin{abstract}
This paper develops a new moment-matching order reduction method for multi-input quadratic-bilinear descriptor systems. The proposed approach accommodates multiple inputs acting on both the differential and algebraic equations and constructs separate projection spaces associated with different input channels and input combinations. The method uses a matrix-free algorithm, allowing the reduced-order matrices and tensors to be computed without explicitly assembling or storing the full-order system matrices and tensors. This feature makes the proposed method particularly suitable for large-scale computational fluid dynamics problems. Numerical experiments carried out for two-dimensional flow problems demonstrate that the resulting reduced-order models accurately reproduce the transient responses of the full-order models under multiple time-varying inputs.
\end{abstract}

\section{Introduction}

In many applications, the same physical system must be simulated repeatedly under different input conditions, such as in design, optimization, control, and uncertainty quantification \cite{B2015}. The dynamic systems that describe these physical systems, either directly or through spatial discretization of partial differential equations, can be large-scale, making high-fidelity simulations computationally expensive, particularly when a large number of simulations are required. Reduced-order models (ROMs) address this issue by replacing the original high-dimensional full order model (FOM) with lower-dimensional models that are faster to simulate yet accurately represent the original solution behavior \cite{H2022}. Many FOMs can be described as Quadratic-bilinear (QB) systems. They may be obtained as second-order Taylor approximations about an operating point \cite{P2003}, as exact lifted representations of certain nonlinear systems \cite{A2017i}, or directly as a QB systems \cite{B2018}.

The general form of QB systems can be written as:
\begin{gather}
\vect E\dot{\vect x}
=\vect A\vect x+\vect B\vect u
+\tensor H\left(\vect x,\vect x\right)
+\tensor N\left(\vect x\right)\vect u
\end{gather}
where $\vect x\in\mathbb{R}^{n}$, $\vect u\in\mathbb{R}^{m}$, $\vect E,\vect A\in\mathbb{R}^{n\times n}$, and $\vect B\in\mathbb{R}^{n\times m}$. The quadratic and bilinear operators are defined as $\tensor H:\mathbb{R}^{n}\times\mathbb{R}^{n}\rightarrow\mathbb{R}^{n}$ and $\tensor N:\mathbb{R}^{n}\rightarrow\mathbb{R}^{n\times m}$, respectively.
The matrix $\vect E$ may be singular, and under an appropriate nonsingular equivalence transformation, $\vect E$ can be brought into the block form $\vect E=\operatorname{diag}(\vect I_{11},\vect 0)$. Accordingly, the state vector is partitioned as $\vect x=[\vect x_1^T,\vect x_2^T]^T$, where $\vect x_1$ and $\vect x_2$ denote the differential and algebraic variables, respectively. Depending on the presence of algebraic variables, QB systems can be categorized as ordinary differential equation (ODE) systems or QB descriptor systems, i.e., differential-algebraic equation (DAE) systems. For many descriptor systems from practical applications \cite{P2003,A2017i,B2018,A2017,A2025}, the algebraic equations are linear and the quadratic terms involving purely algebraic variables vanish. Under these conditions, the QB descriptor systems can be transformed into the following form:
\begin{subequations}\label{QBeqn}
\begin{gather}\label{QBeq}
\dot{\vect x}_1
=\vect A_{11}\vect x_1+\vect A_{12}\vect x_2
+\vect B_1\vect u
+\tensor H_{11}\left(\vect x_1,\vect x_1\right)
+\tensor H_{12}\left(\vect x_1,\vect x_2\right)
+\tensor N_1\left(\vect x_1\right)\vect u
+\tensor N_2\left(\vect x_2\right)\vect u\\
\vect 0
=\vect A_{21}\vect x_1+\vect A_{22}\vect x_2+\vect B_2\vect u
\end{gather}
\end{subequations}

There are generally two classes of ROM for QB systems model reduction: trajectory-based methods and moment-matching methods. Trajectory-based methods construct ROMs from representative solution trajectories generated by FOMs. Common approaches includes proper orthogonal decomposition (POD) \cite{K2005}, POD combined with the discrete empirical interpolation method \cite{C2010}, reduced basis methods \cite{P2001}, and trajectory piecewise linear methods \cite{R2003}. Although trajectory-based methods provide an optimal approximation of the training snapshots and have good performance in many applications, they do not explicitly preserve the local input-output behavior of the original system \cite{K2026}. 

In contrast, moment-matching methods construct ROMs by matching the moments of the Volterra transfer functions about selected interpolation points \cite{B2006}. For ODE systems where $\vect E$ is nonsingular, this method constructs the projection spaces $\mathcal W$ and $\mathcal V$, with $\vect x$ approximated as $\vect x\approx\vect V\vect x_r$, such that the resulting reduced-order system 
\begin{gather}\label{ROMtarget}
\vect W^T \vect E\vect V \dot{\vect x}_r
=\vect W^T\left( \vect A\vect V \vect x_r+\vect B\vect u
+\tensor H\left(\vect V \vect x_r,\vect V \vect x_r\right)
+\tensor N\left(\vect V \vect x_r\right)\vect u\right)
\end{gather}
matches the transfer function of the FOM at the selected interpolation points. In Equation~(\ref{ROMtarget}),  $\vect W$ and $\vect V$ are orthonormal bases of $\mathcal W$ and $\mathcal V$, respectively. The method is referred as a one-sided projection when $\mathcal W=\mathcal V$, and as a two-sided projection when $\mathcal W\neq\mathcal V$. For DAE systems, where $\vect E$ is singular, a common approach is to first transform the system into an equivalent ODE system and then construct the ROM, since direct reduction of the DAE system may lead to unbounded approximation errors \cite{G2013}.
The resulting ROMs can accurately capture the input-output behavior of the FOM around the operating point over the dynamic range represented by the selected interpolation points.
Moreover, for nonlinear systems, POD methods may lose accuracy when varying inputs produce responses that are not adequately represented by the training data \cite{S2017}, whereas moment-matching methods do not rely on specific training trajectories and can therefore offer an advantage in such cases. Existing moment-matching ROMs for QB systems include Krylov-subspace-based moment-matching methods \cite{A2016,A2021}, $\mathcal H_2$-optimal projection methods \cite{B2018}, balanced truncation methods \cite{B2024}, and data-driven approaches \cite{G2018}.

In this paper, we consider a class of QB descriptor systems of the form in Equation~(\ref{QBeqn}), with multiple inputs, each of which may act on both the differential and algebraic equations. Our method is applicable to a broad class of descriptor systems, including computational fluid dynamics (CFD) systems \cite{A2017}, nonlinear RC circuits \cite{A2017i}, and constrained mechanical systems \cite{A2025}. 
A closely related work is that of Ilyas Ahmad et al. \cite{A2017}, who developed a ROM for single-input QB descriptor systems that allows the input to act on both the differential and algebraic equations. However, an input applied to the algebraic equations naturally induces an additional source term when the descriptor system is transformed into an equivalent ODE system, resulting in a multi-input QB system, as described in detail in Section~\ref{Conv}. Consequently, their ROM method is applicable to a system with an input in the algebraic equation only if the input remains constant over time.
Ilyas Ahmad et al. \cite{A2016} constructed a single projection space for multi-input QB systems by combining the Krylov subspaces \cite{G1997} associated with all first- and second-order transfer functions \cite{G2011}, including the mixed-input second-order terms, and used this space to construct a ROM. With this construction, the components associated with different input channels or input combinations may influence the projection spaces associated with the individual transfer functions. Such coupling is not required by the moment-matching conditions, which are imposed separately for each transfer function. Moreover, to the best of our knowledge, no numerical example has been reported demonstrating the application of this construction to a multi-input QB descriptor system.
The proposed method constructs the Krylov subspaces associated with the first- and second-order transfer functions of each input independently, while separate Krylov subspaces are constructed for the mixed-input second-order transfer functions. This construction avoids unnecessary coupling between the projection spaces associated with different transfer functions and enables the practical application of moment matching to multi-input QB descriptor systems, as demonstrated by the numerical examples in this work.

Inspired by existing matrix-free ROM approaches for linear systems \cite{X2022, X2023, X2023_2, L2025}, a completely matrix-free algorithm is developed for QB descriptor systems. The Krylov vectors are obtained by solving the linearized full-order system, while the reduced-order quantities are evaluated directly using the full-order operators, thereby avoiding the explicit construction and storage of the full-order matrices and tensors.
This allows the proposed method to be extended to large-scale simulations whose matrices and tensors are computationally difficult to calculate and store. 
For example, in CFD systems, the velocity field corresponds to the differential variables, while the pressure field and boundary values can be treated as algebraic variables. After spatial discretization, the state vector $x$ consists of the field variables defined at the computational cells and boundary faces, resulting in system matrices of dimension $\mathbb{R}^{n \times n}$, where $n$ denotes the total number of degrees of freedom. More importantly, representing the quadratic terms explicitly requires a third-order $\tensor H\in\mathbb{R}^{n\times n\times n}$, making its direct construction and storage impractical for large-scale CFD systems. The matrix-free formulation avoids the explicit construction and storage of these large-scale operators and therefore makes the ROM in this study more suitable for large-scale CFD applications.

The remainder of this paper is organized as follows. Section \ref{Conv} converts the QB descriptor systems into equivalent ODE systems and Section \ref{ROMDev} develops the ROM for the resulting ODE systems. Section \ref{CompIssue} presents a completely matrix-free procedure for ROM construction without explicitly constructing the system matrices or tensors. In Section \ref{NumEx} we provide results from some CFD tests to demonstrate the accuracy and effectiveness of the proposed method.

\section{Model Development}\label{MD}

\subsection{Conversion of DAEs to ODE systems}\label{Conv}
As shown in Appendix \ref{DS}, the term involving $\vect A_{22}$ in Equation (\ref{QBeqn}) can be eliminated from the algebraic equation through a variable transformation and elimination. For notational simplicity, the resulting DAEs used for the ROM construction are written as:
\begin{subequations}\label{DAE}
\begin{gather}
 \dot{\vect x}_1=\vect A_{11}\vect x_1+\vect A_{12}\vect x_2
  +\vect B_1\vect u
  +\tensor H_{11}\left(\vect x_1, \vect x_1\right)
  +\tensor H_{12}\left(\vect x_1, \vect x_2\right)
  +\tensor N_1\left(\vect x_1\right)\vect u
  +\tensor N_2\left(\vect x_2\right)\vect u\label{DE}
  +\tensor D_1\left(\vect u, \vect u\right)\\
 \vect 0=\vect A_{21}\vect x_1+\vect B_2\vect u \label{AE}
\end{gather}
\end{subequations}

For CFD problems, at the discrete level, the transformation described in Appendix \ref{DS} eliminates the boundary values by incorporating their contributions as equivalent source terms in the cells adjacent to the boundaries. Specifically, the boundary-condition equations are substituted into the discretized governing equations for the adjacent cells, so that the boundary values no longer appear explicitly as algebraic variables. Their effects on the solution are instead preserved through the resulting source terms.

Following the variable transformation proposed in \cite{H2008,A2017}, we make the following change of variable to cancel the $\vect {u}$ term in Equation (\ref{AE}):
\begin{gather}
 \vect y_1=\vect x_1+\vect \Gamma_1 \vect u, \quad 
 \vect y_2=\vect x_2, \quad 
 \vect \Gamma_1 
  =\vect A_{12}\left(\vect A_{21}\vect A_{12}\right)^{-1}\vect B_2
\end{gather}
This requires the assumption of non-singularity of $\vect A_{21}\vect A_{12}$:
\begin{assumption}\label{as2}
$\vect A_{21}\vect A_{12}$ is non-singular.
\end{assumption}
The Equation (\ref{DAE}) becomes:
\begin{subequations}
\begin{gather}
\begin{split}\label{DE1}
 \dot{\vect y}_1-\vect \Gamma_1 \dot{\vect u}=
  \vect A_{11}\vect y_1-\vect A_{11}\vect \Gamma_1 \vect u
  &+\vect A_{12}\vect y_2+\vect B_1\vect u\\&+
  \tensor H_{11}\left(\vect y_1, \vect y_1\right)
  +\tensor H_{12}\left(\vect y_1, \vect y_2\right)
  +\tensor N'_1\left(\vect y_1\right)\vect u
  +\tensor N'_2\left(\vect y_2\right)\vect u
  +\tensor D'_1\left(\vect u, \vect u\right)
\end{split}\\
 \vect 0=\vect A_{21}\vect y_1\label{AE1}
\end{gather}
\end{subequations}
where,
\begin{subequations}\label{qterm1}
\begin{gather}
 \tensor N'_1\left(\vect y_1\right)\vect u
  =-\tensor H_{11}\left(\vect \Gamma_1 \vect u, \vect y_1\right)
  -\tensor H_{11}\left(\vect y_1, \vect \Gamma_1 \vect u\right)
  +\tensor N_1\left(\vect y_1\right)\vect u \label{qterm11},\\
 \tensor N'_2\left(\vect y_2\right)\vect u=
  -\tensor H_{12}\left(\vect \Gamma_1 \vect u, \vect y_2\right)
  +\tensor N_2\left(\vect y_2\right)\vect u\label{qterm12},\\
 \tensor D'_1\left(\vect u, \vect u\right)
  =\tensor H_{11}\left(
  \vect \Gamma_1 \vect u,\vect \Gamma_1 \vect u\right)
  -\tensor N_1\left(\vect \Gamma_1 \vect u\right)\vect u 
  +\tensor D_1\left(\vect u, \vect u\right)\label{qterm13}
\end{gather}
\end{subequations}

Taking time derivative for $\vect A_{21}\vect y_1$, which vanishes because Equation (\ref{AE1}) is an algebraic constraint that must be satisfied at all times, gives:
\begin{subequations}
\begin{gather}
\begin{split}\label{Cancely2_1}
    \vect 0=\dot{\vect A_{21}\vect y_1}=\vect A_{21}\vect A_{12}\vect y_2
  &+\vect A_{21}\bigl(\vect A_{11}\vect y_1
  -\vect A_{11}\vect \Gamma_1 \vect u+\vect B_1\vect u \\
 &+\tensor H_{11}\left(\vect y_1, \vect y_1\right)
  +\tensor H_{12}\left(\vect y_1, \vect y_2\right)
  +\tensor N'_1\left(\vect y_1\right)\vect u
  +\tensor N'_2\left(\vect y_2\right)\vect u
  +\tensor D'_1\left(\vect u, \vect u\right)\bigr)
  +\vect A_{21}\vect \Gamma_1 \dot{\vect u}
\end{split}\\
\begin{split}\label{Cancely2_2}
 \vect y_2=-\left(\vect A_{21}\vect A_{12}\right)^{-1}&\vect A_{21}
  \bigl(\vect A_{11}\vect y_1-\vect A_{11}\vect \Gamma_1 \vect u
  +\vect B_1\vect u \\
+\tensor H_{11}&\left(\vect y_1, \vect y_1\right)
  +\tensor H_{12}\left(\vect y_1, \vect y_2\right)
  +\tensor N'_1\left(\vect y_1\right)\vect u
  +\tensor N'_2\left(\vect y_2\right)\vect u
  +\tensor D'_1\left(\vect u, \vect u\right)\bigr)
  -\left(\vect A_{21}\vect A_{12}\right)^{-1}\vect B_2\dot{\vect u}
\end{split} 
\end{gather}
\end{subequations}
Following the projection-based elimination procedure proposed by Heinkenschloss et al. \cite{H2008}, substituting Equation (\ref{Cancely2_2}) into $\vect A_{12}\vect y_2$ term in Equation (\ref{DE1}), we obtain:
\begin{gather}
 \dot{\vect y}_1=\vect \Pi \left(\vect A_{11}\vect y_1
  -\vect A_{11}\vect \Gamma_1 \vect u+\vect B_1\vect u
  +\tensor H_{11}\left(\vect y_1, \vect y_1\right)
  +\tensor H_{12}\left(\vect y_1, \vect y_2\right)
  +\tensor N'_1\left(\vect y_1\right)\vect u
  +\tensor N'_2\left(\vect y_2\right)\vect u
  +\tensor D'_1\left(\vect u, \vect u\right)\right) \label{DE2}
\end{gather}
where:
\begin{gather}
 \vect \Pi =\vect I_{11} 
  -\vect A_{12}\left(\vect A_{21}\vect A_{12}\right)^{-1}\vect A_{21}
\end{gather}
Substitute $\vect y_2$ in $\tensor H_{12}(\vect y_1, \vect y_2)$ and $\tensor N'_2(\vect y_2)\vect u$,
\begin{subequations}
\begin{gather}
 \tensor H_{12}\left(\vect y_1, \vect y_2\right)=
  -\tensor H_{12}\left(\vect y_1,\left(\vect A_{21}\vect A_{12}\right)^{-1}\vect A_{21}\vect A_{11}\vect y_1\right)
  -\tensor H_{12}\left(\vect y_1,\vectg \Gamma_2 \vect u\right)
  -\tensor H_{12}\left(\vect y_1,\vect \Gamma_3 \dot{\vect u}\right)+\mathcal O\left(\vect r^3\right),\\
 \tensor N'_{2}\left(\vect y_2\right)\vect u=
  -\tensor N'_{2}\left(\left(\vect A_{21}\vect A_{12}\right)^{-1}\vect A_{21}
  \vect A_{11}\vect y_1\right)\vect u
  -\tensor N'_{2}\left(\vectg \Gamma_2 \vect u\right)\vect u
  -\tensor N'_{2}\left(\vect \Gamma_3 \dot{\vect u}\right)\vect u+\mathcal O\left(\vect r^3\right)
\end{gather}
\end{subequations}
where,
%\begin{subequations}
\begin{gather}
 	\vect \Gamma_2=\left(\vect A_{21}\vect A_{12}\right)^{-1}\vect A_{21}\left(-\vect A_{11}\vect \Gamma_1 +\vect B_1\right),\quad
 	\vect \Gamma_3=\left(\vect A_{21}\vect A_{12}\right)^{-1}\vect B_2
\end{gather}
%\end{subequations}

Let,
\begin{subequations}\label{qterm2}
\begin{gather}
\tensor H\left(\vect y_1,\vect y_1\right)
  =\tensor H_{11}\left(\vect y_1,\vect y_1\right)
  -\tensor H_{12}\left(\vect y_1,\left(\vect A_{21}\vect A_{12}\right)^{-1}\vect A_{21}\vect A_{11}\vect y_1\right),\label{qterm21}\\
\tensor N\left(\vect y_1\right) \vect u
  =\tensor N'_1\left(\vect y_1\right) \vect u
  -\tensor H_{12}\left(\vect y_1,\vectg \Gamma_2 \vect u\right)
  -\tensor N'_{2}\left(\left(\vect A_{21}\vect A_{12}\right)^{-1}\vect A_{21}
  \vect A_{11}\vect y_1\right)\vect u,\quad
 \tensor M\left(\vect y_1\right) \dot{\vect u}
  =-\tensor H_{12}\left(\vect y_1,
  \vect \Gamma_3 \dot{\vect u}\right)\label{qterm23},\\
\tensor D\left(\vect u,\vect u\right)
  =\tensor D'_1\left(\vect u, \vect u\right)
  -\tensor N'_{2}\left(\vectg \Gamma_2 \vect u\right)\vect u,\quad
\tensor K\left(\vect u,\dot{\vect u}\right)
  =-\tensor N'_{2}\left(\vect \Gamma_3 \dot{\vect u}\right)\vect u\label{qterm24}
\end{gather}
\end{subequations}
Equation (\ref{DE2}) becomes the following, where $\mathcal O(\vect r^3)$ collects terms of total degree three and higher in $\vect y_1$, $\vect u$, and $\dot{\vect u}$:
\begin{gather}
 \dot{\vect y}_1
  =\vect \Pi \left(\vect A_{11}\vect y_1-\vect A_{11}\vect \Gamma_1 \vect u
  +\vect B_1\vect u+\tensor H\left(\vect y_1, \vect y_1\right)
  +\tensor N\left(\vect y_1\right) \vect u
  +\tensor D\left(\vect u, \vect u\right)
  +\tensor M\left(\vect y_1\right)\dot{\vect u}
  +\tensor K\left(\vect u, \dot{\vect u}\right)
  +\mathcal O\left(\vect r^3\right)\right) \label{DE3}
\end{gather}
The $\mathcal O(\vect r^3)$ term can be neglected in the ROM construction because it does not contribute to the first- and second-order transfer functions considered in this study, as will be shown in Section \ref{ROMDev}.

\begin{prop}
$\operatorname{Im}(\vect \Pi)= \operatorname{ker}(\vect A_{21})$, and $\vect \Pi^2=\vect \Pi$. Moreover, there exists a symmetric positive-definite matrix $\vect G$ such that the operator $\vect \Pi$ is self-adjoint with respect to the inner product $\left\langle \vect x,\vect y\right\rangle_{\vect G}:=\vect x^T\vect G\vect y$. Equivalently, $\vect \Pi^T\vect G=\vect G\vect \Pi$.
\end{prop}
\begin{proof}
Since:
\begin{gather}
 \vect A_{21}\vect \Pi
 =\vect A_{21}-\vect A_{21}\vect A_{12}\left(\vect A_{21}
 \vect A_{12}\right)^{-1}\vect A_{21}=\vect 0
\end{gather}
Hence $\operatorname{Im}(\vect \Pi)\subseteq \operatorname{ker}(\vect A_{21})$. On other hand, for any $\vect y_1 \in \operatorname{ker}(\vect A_{21})$,
\begin{gather}\label{eq2}
 \vect y_1=\vect y_1 -\vect A_{12}
  \left(\vect A_{21}\vect A_{12}\right)^{-1}\vect A_{21}\vect y_1
  =\vect \Pi \vect y_1 \in \operatorname{Im}\left(\vect \Pi\right)
\end{gather}
Hence $\operatorname{ker}(\vect A_{21})\subseteq\operatorname{Im}(\vect \Pi)$, and thus, $\operatorname{Im}(\vect \Pi)=\operatorname{ker}(\vect A_{21})$. This proves the first statement. For the second statement:
\begin{gather}
 \vect \Pi^2=\left(\vect I_{11}-\vect A_{12}
  \left(\vect A_{21}\vect A_{12}\right)^{-1}\vect A_{21}\right)^2
  =\vect I_{11}-2\vect A_{12}\left(\vect A_{21}\vect A_{12}\right)
  ^{-1}\vect A_{21}+\left(\vect A_{12}
  \left(\vect A_{21}\vect A_{12}\right)^{-1}\vect A_{21}\right)^2=\vect \Pi
\end{gather}
Since $\vect \Pi$ is idempotent, it is diagonalizable. Let
\begin{gather}
\vect \Pi=\vect S^{-1}
\begin{bmatrix}
\vect I & \vect 0\\
\vect 0 & \vect 0
\end{bmatrix}
\vect S
\end{gather}
Define $\vect G=\vect S^{T}\vect S$, which is symmetric positive definite since $\vect S$ is nonsingular. Then, $\vect \Pi^T\vect G=\vect G\vect \Pi$.
\end{proof}

%\begin{remark}
%	The construction of the inner product $\left\langle \cdot,\cdot \right\rangle_{\vect G}$ for CFD problems is discussed in Section \ref{G_FVM}, without explicitly using the structures of $\vect A_{12}$ and $\vect A_{21}$. For general problems, if $\vect A_{12}$ and $\vect A_{21}$ are available, one may seek symmetric positive-definite matrices $\vect G_1$ and $\vect G_2$ satisfying $\vect G_1\vect A_{12}=-\vect A_{21}^{T}\vect G_2$, which can be formulated as a semidefinite feasibility problem. The resulting $\vect G_1$  can then be used to define the weighted inner product and satisfies $\vect \Pi^T\vect G_1=\vect G_1\vect \Pi$, implying that $\vect \Pi$ is self-adjoint with respect to this inner product.
%\end{remark}

The proposition shows that $\vect \Pi$ is the $\vect G$-orthogonal projector onto $\operatorname{Im}(\vect \Pi)$. Consequently, there exists $\vect \Phi $ whose columns form a $\vect G$-orthonormal basis of $\operatorname{Im}(\vect \Pi)$, or equivalently $\vect \Phi^T \vect G\vect \Phi =\vect I$, such that $\vect \Pi=\vect \Phi \vect \Phi^T \vect G$. Since $\vect y_1\in \operatorname{ker}(\vect A_{21})=\operatorname{Im}(\vect \Pi)$, there is $\vect z$ such that $\vect y_1=\vect \Phi \vect z$. Thus, $\vect z=\vect \Phi^T \vect G \vect \Phi\vect z= \vect \Phi^T \vect G \vect y_1$ and $\vect \Phi^T \vect G \vect \Pi=\vect \Phi^T \vect G \vect \Phi\vect \Phi^T \vect G=\vect \Phi^T\vect G$. By multiplying $\vect \Phi^T \vect G$ on the left, and drop the $\mathcal O(\vect r^3)$ term, Equation (\ref{DE3}) becomes:
\begin{gather}
 \dot{\vect z}=\vect {\widetilde A} \vect z 
  + \vect {\widetilde B} \vect u 
  + \tensor {\widetilde H}\left(\vect z, \vect z\right)
  + \tensor {\widetilde N}\left(\vect z\right ) \vect u 
  + \tensor {\widetilde D}\left(\vect u, \vect u\right)
  + \tensor {\widetilde M}\left(\vect z\right)\dot{\vect u} 
  + \tensor {\widetilde K}\left(\vect u, \dot{\vect u}\right) \label{DE4}
\end{gather}
Where:
\begin{subequations}
\begin{gather}
 \vect {\widetilde{A}}=\vect \Phi^T\vect G\vect A_{11} \vect \Phi, \quad
 \vect {\widetilde{B}}=\vect \Phi^T \vect G
  \left(-\vect A_{11} \vect\Gamma_1 +\vect B_1\right),\\
 \tensor {\widetilde H}\left(\vect z, \vect z'\right)
  =\vect \Phi^T \vect G\tensor H
  \left(\vect \Phi\vect z, \vect \Phi\vect z'\right),\\
  \tensor {\widetilde N}\left(\vect z\right)
  =\vect \Phi^T \vect G\tensor N\left(\vect \Phi\vect z\right), \quad
 \tensor {\widetilde M}\left(\vect z\right)
  =\vect \Phi^T \vect G\tensor M\left(\vect \Phi\vect z\right),\\
  \tensor {\widetilde D}\left(\vect u, \vect u\right)
  =\vect \Phi^T \vect G\tensor D\left(\vect u, \vect u\right),\quad
  \tensor {\widetilde K}\left(\vect u, \dot{\vect u}\right)
  =\vect \Phi^T \vect G\tensor K\left(\vect u, \dot{\vect u}\right)
\end{gather}
\end{subequations}

\begin{remark}\label{rem}
Based on the analysis in this section, suppose that the solution of the following linear equations are $[\vect x_1^T,\vect x_2^T]^T=[\vectg \xi^T,\vectg \eta^T]^T$, under arbitrary $\sigma$, $\vect u$ and $\vect s$, where $\vect s$ has the same dimension of $\vect x_1$:
	\begin{subequations}
	\begin{gather}
 	\vect 0 =\vect A_{11}\vect x_1+\vect A_{12}\vect x_2
  	+\left(\vect B_1-\sigma\vect\Gamma_1\right)\vect u
  	+\vect s-\sigma \vect x_1\\
 	\vect 0=\vect A_{21}\vect x_1+\vect B_2\vect u
	\end{gather}
	\end{subequations}
and let $\vect z=\vectg \zeta$ be the solution of:
	\begin{gather}
 	\vect 0 =\vect {\widetilde A} \vect z 
  	+ \vect {\widetilde B} \vect u 
  	+ \vect \Phi^T \vect G\vect s-\sigma \vect z
	\end{gather}
The following equations hold:
	\begin{gather}
	\vectg \zeta = \vect \Phi^T \vect G\left(\vectg \xi 
  	+ \vect \Gamma_1 \vect u\right),\qquad
 	\vect \Phi\vectg \zeta = \vectg \xi 
  	+ \vect \Gamma_1 \vect u
  	\end{gather}
These identities will be useful in the construction of the Krylov subspaces.
\end{remark}

\subsection{ROM development}\label{ROMDev}

For the multi-input case, Ilyas Ahmad et al. \cite{A2016} constructed a single projection space using a moment-matching approach for QB systems of the following form to obtain a ROM.
\begin{gather}
 \dot{\vect x}=\vect A\vect x
  +\tensor H\left(\vect x, \vect x\right)
  +\vect B \vect u
  +\tensor N \left(\vect x\right)\vect u
\end{gather}
For each interpolation point $\sigma_i$, the projection space is constructed using:
\begin{gather}
\vect V_1^{i}=\left(\sigma_i \vect I-\vect {A}\right)^{-1}
  \vect B,\quad
\tensor V_2^{i}=\left(2\sigma_i \vect I-\vect {A}\right)^{-1}
  \operatorname{sym}\left(\tensor {H}\left(\vect V_1^{i}, \vect V_1^{i}\right)+\tensor N\left(\vect V_1^{i}\right)\right)
\end{gather}
where the multilinear operators $\tensor H$ and $\tensor N$ are extended to matrix arguments according to $\tensor H(\vect V,\vect W)_{:,p,q}=\tensor H(\vect v_p,\vect w_q)$ and $\tensor N(\vect V)_{:,:,p}=\tensor N(\vect v_p)$, where $\vect v_p$ and $\vect w_q$ denote the $p$-th and $q$-th columns of $\vect V$ and $\vect W$, respectively, with $1\leq p,q\leq m$, and $m$ is the number of manipulated variables. The operator $\operatorname{sym}(\cdot)$ denotes symmetrization with respect to the second and third modes and is defined as:
\begin{gather}
\operatorname{sym}\left(\tensor T\right)_{:,p,q}
=\frac{1}{2}\left(\tensor T_{:,p,q}+\tensor T_{:,q,p}\right)	
\end{gather}
The Krylov subspace is constructed as $\mathcal{V}=\sum_{i=1}^r\operatorname{span}\{\vect V_1^{i},\tensor V_2^{i}\}$,
where the span of a tensor is understood as the span of all its vectors along the first dimension, i.e., the vectors $(\tensor V_2^i)_{:,p,q}$ for all $p$ and $q$. Let $\vect V$ be the matrix whose columns form an orthonormal basis of $\mathcal{V}$. The reduced-order operators are constructed as \cite{A2016}:
\begin{gather}
\vect A_{r}=\vect{V}^T\vect A \vect{V},\quad
\vect B_{r}=\vect{V}^T\vect B,\quad
\tensor H_{r}\left(\vect x_{r}, \vect x_{r}\right)=\vect{V}^T\tensor H\left(\vect V \vect x_{r},\vect V\vect x_{r}\right),\quad
\tensor N_{r}\left(\vect x_{r}\right)=\vect{V}^T\tensor N \left(\vect V \vect x_{r}\right)
\end{gather}
The construction matches the first- and second-order transfer functions at $\sigma_i$ and ($\sigma_i$, $\sigma_i$), respectively, for all combinations of input channels:
\begin{gather}
 \vect G_{r,p}^{\left(1\right)}\left(\sigma_i\right)=\vect G_p^{\left(1\right)}\left(\sigma_i\right),\quad
 \vect G_{r,pq}^{\left(2\right)}\left(\sigma_i,\sigma_i\right)=\vect G_{pq}^{\left(2\right)}\left(\sigma_i,\sigma_i\right)
\end{gather}

With this construction, the vectors associated with different input channels or input combinations are included in the same projection space, causing the reduced operators to contain cross terms between these vectors. For example, let $\vect v=(\vect V_1^i)_{:,p}$, which is the vector obtained from the first-order construction for the input $u_p$, and let $\vect w$ be either $(\vect V_1^j)_{:,q}$, obtained from the first-order construction for $u_q$, or $(\tensor V_2^j)_{:, q, q'}$, obtained from the second-order construction for the input pair $(u_q, u_{q'})$. Since $\vect v$ and $\vect w$ are both included in the projection space, the reduced matrix $\vect A_r=\vect V^T\vect A\vect V$ contains the cross term $\vect v^T\vect A\vect w$. Consequently, terms associated with different input channels or input combinations are coupled through the projection, even though the corresponding vectors are constructed independently in the moment-matching procedure.

To avoid such coupling, the proposed method constructs the Krylov subspaces associated with the first- and second-order transfer functions of each input independently, while separate Krylov subspaces are constructed for the mixed-input second-order transfer functions. 
This construction avoids unnecessary coupling between the projection spaces associated with different transfer functions. This approach helps avoid potential numerical issues that may arise when constructing and running a ROM with larger matrices and tensors, thereby enabling the practical application of moment matching to multi-input QB descriptor systems.
The following theorem shows that, for one-sided projection, the proposed construction preserves the same interpolation conditions as the Krylov-subspace construction by Ilyas Ahmad et al. \cite{A2016}. The formulation in Theorem \ref{theo1} is considered because, by treating both $\vect u$ and $\dot{\vect u}$ as inputs, Equation (\ref{DE4}) can be regarded as the ODE system given by Equation (\ref{proveFOM}) in Theorem \ref{theo1}:

\begin{theorem} \label{theo1} 
Suppose there is a finite-dimensional quadratic-bilinear ODE system of multiple manipulated variables:
\begin{gather}
 \dot{\vect x}=\vect A\vect x
  +\tensor H\left(\vect x, \vect x\right)
  +\sum_{1\leq p \leq m} \vect b_p u_p
  +\sum_{1\leq p \leq m}\vect N_pu_p \vect x
  +\sum_{1\leq p, q\leq m} {\vect d_{pq}}u_pu_q\label{proveFOM}
\end{gather}
Let $\sigma_i$ be the real interpolation points such that $(\sigma_i \vect I-\vect A)$ and $(2\sigma_i \vect I-\vect A)$ are invertible, $i\in \{1,2,\ldots,r\}$. Let:
\begin{subequations}
\begin{gather}
\vectg \zeta_{p}^{i}=\left(\sigma_i \vect I-\vect {A}\right)^{-1}
  \vect b_p\\
\vectg \zeta_{pq}^{i}=\frac{1}{2}\left(2\sigma_i \vect I-\vect {A}\right)^{-1}
  \left(\tensor {H}\left(\vectg \zeta_{p}^{i}, \vectg \zeta_{q}^{i}\right)+\vect N_p \vectg \zeta_{q}^{i}+\vect d_{pq}
   + \tensor {H}\left(\vectg \zeta_{q}^{i}, \vectg \zeta_{p}^{i}\right)
   +\vect N_q \vectg \zeta_{p}^{i}+\vect d_{qp}\right)
\end{gather}
\end{subequations}
For any  $1\leq p \leq q\leq m$. The Krylov subspaces are constructed as follows:
\begin{gather}
 \mathcal{V}_p=
  \sum_{i=1}^r\operatorname{span}
  \left\{\vectg \zeta_{p}^{i},\vectg \zeta_{pp}^{i}\right\},\quad
  \mathcal{V}_{pq}=
  \sum_{i=1}^r\operatorname{span}
  \left\{\vectg {\zeta}_{pq}^{i}\right\} \quad \textnormal{for }p< q
\end{gather}
Let $\vect{V}_p$ and $\vect{V}_{pq}$ be matrices whose columns form orthonormal bases of $\mathcal{V}_p$ and $\mathcal{V}_{pq}$, respectively. The reduced-order operators are of the form, for each $p=1,\dots, m$:
\begin{subequations}
\begin{gather}
\vect A_{r,p}=\vect{V}_p^T\vect A \vect{V}_p,\quad
\tensor H_{r,p}\left(\vect x_{r,p}, \vect x_{r,p}\right)=\vect{V}_p^T\tensor H\left(\vect V_p \vect x_{r,p},\vect V_p\vect x_{r,p}\right),\\
\vect b_{r,p}=\vect{V}_p^T\vect b_p,\quad
\vect N_{r,p}=\vect{V}_p^T\vect N_p \vect{V}_p,\quad
\vect d_{r,p}=\vect{V}_p^T\vect d_{pp} 
\end{gather}
\end{subequations}
and for any $1\leq p < q\leq m$:
\begin{subequations}
\begin{gather}
\vect A_{r,pq}=\vect{V}_{pq}^T\vect A \vect{V}_{pq},\quad
\tensor H_{r,pq}\left(\vect x_{r,p}, \vect x_{r,q}\right)=\frac12 \vect V_{pq}^T\left(\tensor H\left(\vect V_p \vect x_{r,p},\vect V_q\vect x_{r,q}\right)+\tensor H\left(\vect V_q\vect x_{r,q},\vect V_p\vect x_{r,p}\right)\right),\\
\vect N_{r,pq}=\vect{V}_{pq}^T\vect N_p \vect{V}_q,\quad
\vect N_{r,qp}=\vect{V}_{pq}^T\vect N_q \vect{V}_p,\quad
\vect d_{r,pq}=\frac12\vect{V}_{pq}^T\left(\vect d_{pq}+\vect d_{qp}\right)
\end{gather}
\end{subequations}
Assume further that the interpolation points $\sigma_i$ make the reduced resolvent matrices $(\sigma_i\vect I-\vect A_{r,p})$, $(2\sigma_i\vect I-\vect A_{r,p})$, and $(2\sigma_i\vect I-\vect A_{r,pq})$ nonsingular, for $i\in \{1,2,\ldots,r\}$. The ROM systems:
\begin{subequations}\label{rom}
\begin{gather}
 \dot{\vect x}_{r,p}=\vect A_{r,p}\vect x_{r,p}
  +\tensor H_{r,p}\left(\vect x_{r,p}, \vect x_{r,p}\right)
  +\vect b_{r,p} u_{p}
  +\vect N_{r,p}u_p \vect x_{r,p}
  +{\vect d_{r,p}}u_p^2 \label{rom1}\\
 \dot{\vect x}_{r,pq}=\vect A_{r,pq}\vect x_{r,pq}
  +\tensor H_{r,pq}\left(\vect x_{r,p},\vect x_{r,q}\right)
  +\frac{1}{2}\left(\vect N_{r,pq}u_p \vect x_{r,q}
  +\vect N_{r,qp}u_q \vect x_{r,p}\right)
  +{\vect d_{r,pq}}u_pu_q\label{rom2}\\
 \vect {\hat x}=\sum_{1\leq p\leq m}\vect{V}_p\vect x_{r,p}+\sum_{1\leq p < q\leq m}\vect{V}_{pq}\vect x_{r,pq}\label{rom3}
\end{gather}
\end{subequations}
matches the following transfer functions for the state variables, for $i\in \{1,2,\ldots,r\}$ and any $1\leq p, q\leq m$:
\begin{gather}
 \vect G_{r,p}^{(1)}\left(\sigma_i\right)=\vect G_p^{(1)}\left(\sigma_i\right),\quad
 \vect G_{r,pq}^{(2)}\left(\sigma_i,\sigma_i\right)=\vect G_{pq}^{(2)}\left(\sigma_i,\sigma_i\right)
\end{gather}
\end{theorem}

\begin{proof}
The right-hand side of Equation (\ref{proveFOM}) is a polynomial, and by the assumption of finite dimension, is analytic. For inputs $u_p(t)=\varepsilon \bar u_p(t)$, $1\leq p\leq m$, the state variable response can be written as local expansion: 
 $\vect x=\varepsilon \vect x^{(1)}+\varepsilon^2 \vect x^{(2)}+\mathcal O(\varepsilon^3)$.
 Put back into Equation (\ref{proveFOM}) and collect the $\mathcal O(\varepsilon)$ term:
 \begin{gather}
  \dot{\vect x}^{(1)}=\vect A \vect x^{(1)}+\sum_{1\le p\le m}\vect b_p\bar u_p
 \end{gather}
 Take the Laplace transform, the transfer function is:
 \begin{gather}
 \vect G_p^{(1)}\left(s\right)=\vect R\left(s\right)\vect b_p \label{proveG1}
 \end{gather}
 Where $\vect R(s)=(s\vect I-\vect A)^{-1}$. Collect the $\mathcal O(\varepsilon^2)$ term:
 \begin{gather}
 \dot{\vect x}^{(2)}=\vect A\vect x^{(2)}+\tensor H\left(\vect x^{(1)},\vect x^{(1)}\right)+\sum_{1\le p\le m} \vect N_p\bar u_p \vect x^{(1)}+\sum_{1\le p, q\le m}\vect d_{pq}\bar u_p\bar u_q
 \end{gather}
 Apply the first order transfer functions of $p$ and $q$ to the equation, and notice that the second order frequency $s=s_1+s_2$, the symmetrized second order transfer function of $p$ and $q$ is, with $1\le p, q\le m$:
 \begin{gather}
 \begin{split}
 \vect G_{pq}^{(2)}\left(s_1,s_2\right)=\frac12 \vect R\left(s_1+s_2\right)
  &\bigl(\tensor H\left(\vect R\left(s_1\right)\vect b_p,\vect R\left(s_2\right)\vect b_q\right)
   +\tensor H\left(\vect R\left(s_2\right)\vect b_q,\vect R\left(s_1\right)\vect b_p\right)\\&\qquad\qquad\qquad\qquad\qquad
   +\vect N_p \vect R\left(s_2\right)\vect b_q
   +\vect N_q \vect R\left(s_1\right)\vect b_p+\vect d_{pq}+\vect d_{qp}\bigr) 
   \label{proveG2}
 \end{split}
 \end{gather}
 Since $\vectg \zeta_{p}^{i}=(\sigma_i \vect I-\vect {A})^{-1}
  \vect b_p \in \mathcal{V}_{p}$, there exists $\vect x_{r,p}^{i}$ such that $\vectg \zeta_{p}^{i} = \vect{V}_p \vect x_{r,p}^{i}$, thus:
 \begin{gather}
 \left(\sigma_i \vect I-\vect {A}_{r,p}\right)\vect x_{r,p}^{i}
  =\vect{V}_p^T
  \left(\sigma_i \vect I-\vect {A}\right)\vect{V}_p\vect x_{r,p}^{i}
  =\vect{V}_p^T
  \left(\sigma_i \vect I-\vect {A}\right)\vectg \zeta_{p}^{i}
  =\vect{V}_p^T\vect b_p
  =\vect b_{r,p}\\
 \vect G_{r,p}^{(1)}\left(\sigma_i\right)=\vect{V}_p\left(\sigma_i \vect I-\vect {A}_{r,p}\right)^{-1}\vect b_{r,p}=\vect{V}_p\vect x_{r,p}^{i}=\vectg \zeta_{p}^{i}=\vect G_{p}^{(1)}\left(\sigma_i\right) \label{proveGr1p}
 \end{gather}
 Where Equation (\ref{proveG1}) has been used for FOM and ROM systems. Similarly, there exists $\vect x_{r,pp}^{i}$ and $\vect x_{r,pq}^{i}$ such that $\vectg \zeta_{pp}^{i}= \vect{V}_p \vect x_{r,pp}^{i}$ and $\vectg \zeta_{pq}^{i}= \vect{V}_{pq} \vect x_{r,pq}^{i}$, respectively. Then,
%\begin{subequations}
\begin{gather}
 \begin{split}
 \left(2\sigma_i \vect I-\vect {A}_{r,p}\right)\vect x_{r,pp}^{i}
  &=\vect{V}_p^T\left(2\sigma_i \vect I-\vect {A}\right)
  \vect{V}_p\vect x_{r,pp}^{i}
  =\vect{V}_p^T\left(2\sigma_i \vect I-\vect {A}\right)
  \zeta_{pp}^{i}\\
  &=\vect{V}_p^T
  \left(\tensor {H}\left(\vectg \zeta_{p}^{i},\vectg \zeta_{p}^{i}\right)
  +\vect N_p \vectg \zeta_{p}^{i}
  +\vect d_{pp}\right)
  =\tensor H_{r,p}\left(\vect x_{r,p}^i, \vect x_{r,p}^i\right)
  +\vect N_{r,p} \vect x_{r,p}^i
  +{\vect d_{r,p}}
 \end{split}\\
 \begin{split}
 \left(2\sigma_i \vect I-\vect {A}_{r,pq}\right)\vect x_{r,pq}^{i}
  &=\vect{V}_{pq}^T\left(2\sigma_i \vect I-\vect {A}\right)
  \vect{V}_{pq}\vect x_{r,pq}^{i}
  =\vect{V}_{pq}^T\left(2\sigma_i \vect I-\vect {A}\right)
  \zeta_{pq}^{i}\\
  &=\frac{1}{2}\vect{V}_{pq}^T\left(\tensor {H}\left(\vectg \zeta_{p}^{i}, \vectg \zeta_{q}^{i}\right)+\vect N_p \vectg \zeta_{q}^{i}+\vect d_{pq}
   + \tensor {H}\left(\vectg \zeta_{q}^{i}, \vectg \zeta_{p}^{i}\right)
   +\vect N_q \vectg \zeta_{p}^{i}+\vect d_{qp}\right)\\
  &=\tensor H_{r,pq}\left(\vect x_{r,p}^i,\vect x_{r,q}^i\right)
  +\frac{1}{2}\left(\vect N_{r,pq} \vect x_{r,q}^i
  +\vect N_{r,qp} \vect x_{r,p}^i\right)
  +{\vect d_{r,pq}}
 \end{split}
\end{gather}
%\end{subequations}
Using Equation (\ref{proveG2}) for FOM and ROM systems, the second-order transfer functions of the ROM are:
%\begin{subequations}\label{proveGr2}
\begin{align}
    \vect G_{r,pp}^{(2)}\left(\sigma_i,\sigma_i\right)
    &=\vect{V}_p\left(2\sigma_i \vect I-\vect {A}_{r,p}\right)^{-1}
    \left(\tensor H_{r,p}\left(\vect x_{r,p}^i, \vect x_{r,p}^i\right)
  +\vect N_{r,p}\vect x_{r,p}^i
  +{\vect d_{r,p}}\right)=\vect{V}_p\vect x_{r,pp}^{i}
  =\vectg \zeta_{pp}^{i}=\vect G_{pp}^{(2)}\left(\sigma_i,\sigma_i\right)
  \label{proveGr2pp}\\
    \begin{split}
        \vect G_{r,pq}^{(2)}\left(\sigma_i,\sigma_i\right)
        &=\vect{V}_{pq}\left(2\sigma_i \vect I-\vect {A}_{r,pq}\right)^{-1}
        \left(\tensor H_{r,pq}\left(\vect x_{r,p}^i,\vect x_{r,q}^i\right)
  +\frac{1}{2}\left(\vect N_{r,pq} \vect x_{r,q}^i
  +\vect N_{r,qp} \vect x_{r,p}^i\right)
  +{\vect d_{r,pq}}\right) \\ 
  &=\vect{V}_{pq}\vect x_{r,pq}^{i}=\vectg \zeta_{pq}^{i}
  =\vect G_{pq}^{(2)}\left(\sigma_i,\sigma_i\right)
    \end{split} \label{proveGr2pq}
\end{align}
%\end{subequations}
The conclusion follows by Equation (\ref{proveGr1p}), (\ref{proveGr2pp}) and (\ref{proveGr2pq}).
\end{proof}

By regarding Equation (\ref{DE4}) as the ODE system in Theorem \ref{theo1} and treating $\vect u$ and $\dot{\vect u}$ as independent inputs, the resulting system has $2m$ inputs, where $m$ is the dimension of $\vect u$. Therefore, $\vectg \zeta_{p}^{i}$ and $\vectg \zeta_{pq}^{i}$ should, in general, be considered for all $1\leq p,q\leq 2m$. However, since the first-order source term $\vect{\widetilde B}\vect u$ does not contain $\dot{\vect u}$, $\vectg \zeta_{p}^{i}$ vanishes for $p>m$. Consequently, together with the absence of a $\dot{\vect u}\dot{\vect u}$ term in Equation (\ref{DE4}), $\vectg \zeta_{pq}^{i}$ vanishes for $m<p,q\leq 2m$. Based on Theorem \ref{theo1}, the Krylov subspaces are constructed by choosing $r$ interpolation points $\sigma_i$ for each $1\leq p,q\leq m$ as follows:
\begin{subequations}
\begin{gather}
\vectg \zeta_{p}^{i}=\left(\sigma_i-\vect {\widetilde{A}}\right)^{-1}
  \vect {\widetilde{B}} \vect e_p\\
\vectg \zeta_{pq}^{i}=\frac{1}{2}\left(2\sigma_i-\vect {\widetilde{A}}\right)^{-1}
  \left(\tensor {\widetilde H}\left(\vectg \zeta_{p}^{i}, \vectg \zeta_{q}^{i}\right) +\tensor {\widetilde N}\left(\vectg \zeta_{p}^{i}\right)\vect e_q+\vect {\widetilde D}\left(\vect e_p, \vect e_q\right)
  +\tensor {\widetilde H}\left(\vectg \zeta_{q}^{i}, \vectg \zeta_{p}^{i}\right) 
  +\tensor {\widetilde N}\left(\vectg \zeta_{q}^{i}\right)\vect e_p +\vect {\widetilde D}\left(\vect e_q, \vect e_p\right)\right)\quad  \textnormal{for }p\leq q\\
\vectg \zeta_{p\left(q+m\right)}^{i}=\frac{1}{2}
  \left(2\sigma_i-\vect{\widetilde{A}}\right)^{-1}\left(
  \tensor {\widetilde M}\left(\vectg \zeta_{p}^{i} \right)\vect e_q+\tensor {\widetilde K}\left(\vect e_p, \vect e_q\right)\right)
\end{gather}
\end{subequations}
Where $\vect e_p=[0,\ldots,0,1,0,\ldots,0]^T$ is of $m$ dimension, whose $p$th element is $1$ and other elements are $0$. The Krylov subspaces for Equation(\ref{DE4}) is, for $1\leq p\leq m$ and $p<q \leq 2m$:
\begin{gather}
 \mathcal{V}_p=
  \sum_{i=1}^r\operatorname{span}
  \left\{\vectg \zeta_{p}^{i},\vectg \zeta_{pp}^{i}\right\},\quad
 \mathcal{V}_{pq}=
  \sum_{i=1}^r\operatorname{span}
  \left\{\vectg \zeta_{pq}^{i}\right\}
\end{gather}
Let $\vectg \xi_{p}^{i}$ and $\vectg \eta_{p}^{i}$ be the solution of $\vect  x_1$ and $\vect  x_2$, respectively, under $\sigma_i$ and $\vect e_p$, of the following equations, 
\begin{subequations}\label{eq1}
\begin{gather}
 \vect 0=\vect A_{11}\vect x_1+\vect A_{12}\vect x_2
  +\vect B_1\vect e_p-\sigma_i\vect\Gamma_1\vect e_p-\sigma_i \vect x_1\label{eq11}\\
 \vect 0=\vect A_{21}\vect x_1+\vect B_2\vect e_p \label{eq12}
\end{gather}
\end{subequations}
And let $\vectg \xi_{pq}^{i}$ and $\vectg \eta_{pq}^{i}$ be the solution of $\vect x_1$ and $\vect x_2$, respectively, of the following equations, for any $1\leq p \leq q\leq m$:
\begin{subequations}
\begin{align}
\begin{split}
 \vect 0=\vect A_{11}\vect x_1+\vect A_{12}\vect x_2-2\sigma_i \vect x_1+&\frac12  \bigl(\tensor H\left(\vectg \xi_{p}^{i}+\vect \Gamma_1 \vect e_p, \vectg \xi_{q}^{i}+\vect \Gamma_1 \vect e_q\right)
  +\tensor N\left(\vectg \xi_{p}^{i}+\vect \Gamma_1 \vect e_p\right) \vect e_q
  +\tensor D\left(\vect e_p, \vect e_q\right)\\
  &+\tensor H\left(\vectg \xi_{q}^{i}+\vect \Gamma_1 \vect e_q, \vectg \xi_{p}^{i}+\vect \Gamma_1 \vect e_p\right)
  +\tensor N\left(\vectg \xi_{q}^{i}+\vect \Gamma_1 \vect e_q\right) \vect e_p
  +\tensor D\left(\vect e_q, \vect e_p\right)\bigr)\\
\end{split}\label{eq21}\\
 &\vect 0=\vect A_{21}\vect x_1\label{eq22}
\end{align}
\end{subequations}
And let $\vectg \xi_{p(q+m)}^{i}$ and $\vectg \eta_{p(q+m)}^{i}$ be the solution of $\vect x_1$ and $\vect x_2$, respectively, of the following equations, for any $1\leq p,q\leq m$:
\begin{subequations}
\begin{gather}
 \vect 0=\vect A_{11}\vect x_1+\vect A_{12}\vect x_2-2\sigma_i \vect x_1
 +\frac12\left(\tensor M\left(\vectg \xi_{p}^{i}+\vect \Gamma_1 \vect e_p\right)\vect e_q
  +\tensor K\left(\vect e_p, \vect e_q\right)\right)\label{eq31}\\
 \vect 0=\vect A_{21}\vect x_1\label{eq32}
\end{gather}
\end{subequations}
Then, based on Remark \ref{rem}, we have the following equation, where $i\in \{1,2,\ldots,r\}$, $1\leq p\leq m$ and $p \leq q\leq 2m$:
\begin{subequations}\label{VWcov}
\begin{gather}
 \vectg \zeta_{p}^{i} = \vect \Phi^T \vect{G}\left(\vectg \xi_{p}^{i} 
  + \vect \Gamma_1 \vect e_p\right),\quad 
 \vect \Phi\vectg \zeta_{p}^{i} = \vectg \xi_{p}^{i} 
  + \vect \Gamma_1 \vect e_p\label{VWcov1}\\
 \vectg \zeta_{pq}^{i} = \vect \Phi^T\vect{G} \vectg \xi_{pq}^{i},\quad
 \vect \Phi \vectg \zeta_{pq}^{i} = \vectg \xi_{pq}^{i}\label{VWcov2}
\end{gather}
\end{subequations}

Define:
\begin{gather}
 \mathcal{W}_p=
 \sum_{i=1}^r\operatorname{span}
 \left\{\vectg \xi_{p}^{i}+\vect \Gamma_1 \vect e_p,\vectg \xi_{pp}^{i}\right\},\quad
 \mathcal{W}_{pq}=
  \sum_{i=1}^r\operatorname{span}
  \left\{\vectg \xi_{pq}^{i}\right\}\label{Kspace}
\end{gather}
Let the $\vect G$-orthonormal basis matrices of $\mathcal{W}_p$ and $\mathcal{W}_{pq}$ be $\vect{W}_p$ and $\vect{W}_{pq}$, respectively. From Equation (\ref{VWcov}), $\mathcal{V}_p$ and $\mathcal{V}_{pq}$ are spanned by the column of $\vect \Phi^T\vect G \vect{W}_p$ and $\vect \Phi^T\vect G \vect{W}_{pq}$, respectively. Since:
 \begin{gather}
 	\left(\vect \Phi^T \vect G \vect{W}_p\right)^T\vect \Phi^T \vect G \vect{W}_p=\vect{W}_p^T \vect G\vect \Phi \vect \Phi^T \vect G \vect{W}_p=\vect{W}_p^T \vect G \vect \Pi \vect{W}_p =\vect{W}_p^T \vect G \vect{W}_p=\vect I
 \end{gather}
 and, similarly $(\vect \Phi^T \vect G \vect{W}_{pq})^T\vect \Phi^T \vect G \vect{W}_{pq} = \vect I$, the columns of $\vect{\Phi}^{T}\vect{G}\vect{W}_p$ and $\vect{\Phi}^{T}\vect{G}\vect{W}_{pq}$ form orthonormal bases for $\mathcal{V}_p$ and $\mathcal{V}_{pq}$, respectively. We therefore define $ \vect{V}_p=\vect{\Phi}^{T}\vect{G}\vect{W}_p$, $\vect{V}_{pq}=\vect{\Phi}^{T}\vect{G}\vect{W}_{pq}$. Consequently, $\vect{W}_p=\vect \Phi \vect{V}_p$, $\vect{W}_{pq}=\vect \Phi \vect{V}_{pq}$. The ROM can be constructed by Theorem \ref{theo1}, for $1\leq p\leq m$:
\begin{subequations}\label{rommatrix1}
\begin{gather}
\vect A_{r,p}=\vect{V}_p^T\vect {\widetilde A} \vect{V}_p=
\vect{V}_p^T\vect \Phi^T\vect G\vect A_{11} \vect \Phi \vect{V}_p=
\left\langle\vect \Phi \vect{V}_p,\vect A_{11} \vect \Phi \vect{V}_p\right\rangle_{\vect G}
= \left\langle\vect W_p,\  \vect A_{11} \vect W_p\right\rangle_{\vect G},\\
\vect b_{r,p}=\vect{V}_p^T\vect {\widetilde B} \vect e_p
=\left\langle\vect W_p,\  \left(-\vect A_{11}\vect \Gamma_1  
  +\vect B_1\right) \vect e_p\right\rangle_{\vect G},\\
\tensor H_{r,p}\left(\vect x_{r,p}, \vect x_{r,p}\right)=
\vect{V}_p^T \tensor {\widetilde H}\left(\vect V_p \vect x_{r,p},\vect V_p\vect x_{r,p}\right)=
\left\langle\vect W_p,\   \tensor {H}\left(\vect W_p \vect x_{r,p},\vect W_p\vect x_{r,p}\right)\right\rangle_{\vect G},\\
\vect N_{r,p}\left(\vect x_{r,p} \right)=\vect{V}_p^T\tensor {\widetilde N} \left(\vect{V}_p \vect x_{r,p}\right)\vect e_p=\left\langle\vect W_p,\  \tensor  N \left(\vect{W}_p \vect x_{r,p}\right)\vect e_p\right\rangle_{\vect G},\\
\vect d_{r,p}=\vect{V}_p^T\tensor {\widetilde D}\left(\vect e_p, \vect e_p\right) =\left\langle\vect W_p,\  \tensor {D}\left(\vect e_p, \vect e_p\right)\right\rangle_{\vect G}
\end{gather}
\end{subequations}
The ROMs for the cross terms could be calculated similarly. For $1\leq p< q\leq m$:
\begin{subequations}\label{rommatrix2}
\begin{gather}
\vect A_{r,pq}=\left\langle\vect W_{pq},\  \vect A_{11} \vect{W}_{pq}\right\rangle_{\vect G},\\
\tensor H_{r,pq}\left(\vect x_{r,p}, \vect x_{r,q}\right)=\frac12 \left\langle\vect W_{pq},\  \tensor H\left(\vect W_p \vect x_{r,p},\vect W_q\vect x_{r,q}\right)+\tensor H\left(\vect W_q\vect x_{r,q},\vect W_p\vect x_{r,p}\right)\right\rangle_{\vect G},\\
\vect N_{r,pq}\left(\vect x_{r,q} \right)=\left\langle\vect W_{pq},\  \tensor N \left(\vect{W}_q \vect x_{r,q}\right)\vect e_p\right\rangle_{\vect G},\quad
\vect N_{r,qp}\left(\vect x_{r,p} \right)=\left\langle\vect W_{pq},\  \tensor N \left(\vect{W}_p \vect x_{r,p}\right)\vect e_q\right\rangle_{\vect G},\\
\vect d_{r,pq}=\frac12\left\langle\vect W_{pq},\ \tensor D\left(\vect e_p, \vect e_q\right)+\tensor D\left(\vect e_q, \vect e_p\right)\right\rangle_{\vect G}
\end{gather}
\end{subequations}
and for $1\leq p\leq m$, $m< q\leq 2m$:
\begin{subequations}\label{rommatrix3}
\begin{gather}
\vect A_{r,pq}=\left\langle\vect W_{pq},\ \vect A_{11} \vect{W}_{pq}\right\rangle_{\vect G},\quad \tensor H_{r,pq}=\vect 0, \\
\vect N_{r,pq}=\vect 0, \quad
\vect N_{r,qp}\left(\vect x_{r,p} \right)=\left\langle\vect W_{pq},\ \tensor M \left(\vect{W}_p \vect x_{r,p}\right)\vect e_{q-m}\right\rangle_{\vect G},\quad
\vect d_{r,pq}=\frac12\left\langle\vect W_{pq},\ \tensor K\left(\vect e_p, \vect e_{q-m}\right)\right\rangle_{\vect G}
\end{gather}
\end{subequations}
where, for conformable matrices or vectors $\vect X$ and $\vect Y$, we use the notation $\langle \vect X,\vect Y\rangle_{\vect G}:=\vect X^T\vect G\vect Y$. The ROM systems can be constructed by Equation (\ref{rom}) in Theorem \ref{theo1}. The ROM has a total of $2m$ inputs, corresponding to $[\vect u^T,\dot{\vect u}^T]^T$. Any terms in Equation (\ref{rom}) that are not explicitly specified above are taken to be zero.

\section{Computational Implementation}\label{CompIssue}
\subsection{Procedure of ROM construction}
Solving Equation (\ref{eq21}) requires the calculation of $(\vect A_{21}\vect A_{12})^{-1}\vect A_{21}\vect A_{11}(\vectg \xi_{p}^{i}+\vect \Gamma_1 \vect e_p)$ term, as it appears in Equation~(\ref{qterm2}). This term can be simplified by the following equation, which is based on Equation (\ref{eq1}):
\begin{gather}
\vect A_{21}\left(\vect A_{11}\left(\vectg \xi_{p}^{i}+\vect \Gamma_1 \vect e_p\right)+\left(-\vect A_{11}\vect \Gamma_1  +\vect B_1\right)\vect e_p+\vect A_{12}\vectg \eta_{p}^{i}\right)
=\sigma_i\vect A_{21}\left(\vectg \xi_{p}^{i}+\vect \Gamma_1 \vect e_p\right)
=\sigma_i\left(\vect A_{21}\vectg \xi_{p}^{i}+\vect B_2\vect e_p\right)=\vect 0
\end{gather}
Multiply by $(\vect A_{21}\vect A_{12})^{-1}$:
\begin{gather}
\begin{split}
\left(\vect A_{21}\vect A_{12}\right)^{-1}\vect A_{21}\vect A_{11}\left(\vectg \xi_{p}^{i}+\vect \Gamma_1 \vect e_p\right)
&=-\left(\vect A_{21}\vect A_{12}\right)^{-1}\vect A_{21}\left(-\vect A_{11}\vect \Gamma_1  +\vect B_1\right)\vect e_p-\left(\vect A_{21}\vect A_{12}\right)^{-1}\vect A_{21}\vect A_{12}\vectg \eta_{p}^{i}\\
&=-\vect \Gamma_2\vect e_p-\vectg \eta_{p}^{i}
\end{split} \label{qterm25}
\end{gather}
By Equations (\ref{qterm1}), (\ref{qterm2}), and (\ref{qterm25}), Equation (\ref{eq21}) becomes:
\begin{gather}
\begin{split}
 \vect 0=\vect A_{11}\vect x_1+\vect A_{12}\vect x_2-2\sigma_i \vect x_1+&\frac12  \bigl(\tensor H_{11}\left(\vectg \xi_{p}^{i}, \vectg \xi_{q}^{i}\right)
  +\tensor H_{12}\left(\vectg \xi_{p}^{i}, \vectg \eta_{q}^{i}\right)
  +\tensor N_1\left(\vectg \xi_{p}^{i}\right)\vect e_q
  +\tensor N_2\left(\vectg \eta_{p}^{i}\right)\vect e_q
  +\tensor D_1\left(\vect e_p, \vect e_q\right)\\
  &\ +\tensor H_{11}\left(\vectg \xi_{q}^{i}, \vectg \xi_{p}^{i}\right)
  +\tensor H_{12}\left(\vectg \xi_{q}^{i}, \vectg \eta_{p}^{i}\right)
  +\tensor N_1\left(\vectg \xi_{q}^{i}\right)\vect e_p
  +\tensor N_2\left(\vectg \eta_{q}^{i}\right)\vect e_p
  +\tensor D_1\left(\vect e_q, \vect e_p\right)\bigr)\\
\end{split}
\end{gather}
And similarly, Equation (\ref{eq31}) becomes:
\begin{gather}
 \vect 0=\vect A_{11}\vect x_1+\vect A_{12}\vect x_2-2\sigma_i \vect x_1
 +\frac12\left(-\tensor H_{12}\left(\vectg \xi_{p}^{i}, \vect \Gamma_3\vect e_{q}\right)
  +\tensor N_2\left(\vect \Gamma_3\vect e_{q}\right)\vect e_p\right)
\end{gather}

$\vect \Gamma_1$ and $\vect \Gamma_3$ is need to be used for the calculation, and could be solved by the following equations:
%\begin{subequations}
\begin{gather}
	\vect A_{21}\vect A_{12}\vect \Gamma_3\vect e_p=\vect B_2\vect e_p\\
	\vect \Gamma_1=\vect A_{12}\vect \Gamma_3
\end{gather}
%\end{subequations}
Algorithm \ref{alg1} summarizes the  procedure of the construction of Krylov subspaces of quadratic-bilinear DAEs.
\begin{algorithm}[H]
\caption{Krylov subspaces construction of quadratic-bilinear DAEs}\label{alg1}
\begin{algorithmic}[1] 
    \State Solve equation $\vect A_{21}\vect A_{12}\vect \Gamma_3\vect e_p=\vect B_2\vect e_p$ for $\vect \Gamma_3\vect e_p$, the $p$th column for $\vect \Gamma_3$, calculate $\vect \Gamma_1$ by $\vect \Gamma_1=\vect A_{12}\vect \Gamma_3$.
    
    \State Solve equations for $\vect x_1$ and $\vect x_2$ to obtain $\vectg \xi_p^i$ and $\vectg \eta_p^i$, for all $1\leq p \leq m$, and all interpolation points $\sigma_i$:
       
    \leavevmode$\begin{aligned}
 		&\vect 0=\vect A_{11}\vect x_1+\vect A_{12}\vect x_2
  		+\vect B_1\vect e_p-\sigma_i\vect\Gamma_1\vect e_p-\sigma_i \vect x_1\\
 		&\vect 0=\vect A_{21}\vect x_1+\vect B_2\vect e_p
	\end{aligned}$

	\State Solve equations for $\vect x_1$ and $\vect x_2$ to obtain $\vectg \xi_{pq}^i$ and $\vectg \eta_{pq}^i$, for all $1\leq p \leq q\leq m$ and all interpolation points $\sigma_i$:
	
    \leavevmode$\begin{aligned}
 		&\vect 0=\vect A_{11}\vect x_1+\vect A_{12}\vect x_2-2\sigma_i \vect x_1+\frac12 \bigl(\tensor H_{11}\left(\vectg \xi_{p}^{i}, \vectg \xi_{q}^{i}\right)
  		+\tensor H_{12}\left(\vectg \xi_{p}^{i}, \vectg \eta_{q}^{i}\right)
  		+\tensor N_1\left(\vectg \xi_{p}^{i}\right)\vect e_q
  		+\tensor N_2\left(\vectg \eta_{p}^{i}\right)\vect e_q
  		+\tensor D_1\left(\vect e_p, \vect e_q\right)\\
  		&\qquad\qquad\qquad\qquad\qquad\qquad\qquad+\tensor H_{11}\left(\vectg \xi_{q}^{i}, \vectg \xi_{p}^{i}\right)
  		+\tensor H_{12}\left(\vectg \xi_{q}^{i}, \vectg \eta_{p}^{i}\right)
  		+\tensor N_1\left(\vectg \xi_{q}^{i}\right)\vect e_p
  		+\tensor N_2\left(\vectg \eta_{q}^{i}\right)\vect e_p
  		+\tensor D_1\left(\vect e_q, \vect e_p\right)\bigr)\\
 		&\vect 0=\vect A_{21}\vect x_1
	\end{aligned}$

	\State Solve the equations for $\vect x_1$ and $\vect x_2$ to obtain $\vectg \xi_{p(q+m)}^i$ and $\vectg \eta_{p(q+m)}^i$, for all $1\leq p,q\leq m$ and all interpolation points $\sigma_i$:
	
    \leavevmode$\begin{aligned}
 		&\vect 0=\vect A_{11}\vect x_1+\vect A_{12}\vect x_2-2\sigma_i \vect x_1
 		+\frac12\left(-\tensor H_{12}\left(\vectg \xi_{p}^{i}, \vect \Gamma_3\vect e_{q}\right)
  		+\tensor N_2\left(\vect \Gamma_3\vect e_{q}\right)\vect e_p\right)\\
 		&\vect 0=\vect A_{21}\vect x_1
	\end{aligned}$
	
	\State Construct Krylov subspaces:
	
	\leavevmode$\begin{aligned}
 	\mathcal{W}_p=
 	\sum_{i=1}^r\operatorname{span}
 	\left\{\vectg \xi_{p}^{i}+\vect \Gamma_1 \vect e_p,\vectg \xi_{pp}^{i}\right\},\quad
 	\mathcal{W}_{pq}=
  	\sum_{i=1}^r\operatorname{span}
  	\left\{\vectg \xi_{pq}^{i}\right\}
	\end{aligned}$. 
	
	\State Use the Gram-Schmidt process with respect to the $\vect G$-inner product to obtain $\vect G$-orthonormal basis matrices $\vect W_p$ and $\vect W_{pq}$ for the Krylov subspaces.
\end{algorithmic}
\end{algorithm}

In many large-scale problems, such as CFD systems, the system matrices and tensors are difficult to calculate and store. Instead, the discretized equations are typically solved iteratively through local coupling between neighboring unknowns. The proposed algorithm is well suited to such applications. Throughout Algorithm \ref{alg1}, the system matrices and tensors need not be explicitly assembled or stored. Instead, the algorithm only requires repeatedly solving the same linearized system with different source terms, which can be solved using the native discretization schemes and solution procedures of a CFD solver.

The ROM matrices and tensors can be calculated by Equations (\ref{qterm1}), (\ref{qterm2}) and (\ref{rommatrix1})--(\ref{rommatrix3}). To evaluate Equation~(\ref{qterm2}), the term $(\vect A_{21}\vect A_{12})^{-1}\vect A_{21}\vect A_{11}\vect w$ needs to be computed, where $\vect w$ denotes a column of either $\vect W_p$ or $\vect W_{pq}$. This term can be obtained by solving the following equation for $\vectg\theta$:
	\begin{gather}
	\vect A_{21}\vect A_{12}\vectg \theta=\vect A_{21}\vect A_{11}\vect w
	\end{gather}
Similarly, $\vect \Gamma_2$ could be obtained by solve equation:
\begin{gather}
 	\vect A_{21}\vect A_{12}\vect \Gamma_2\vect e_p=\vect A_{21}\left(-\vect A_{11}\vect \Gamma_1 +\vect B_1\right)\vect e_p
\end{gather}
The procedure of the calculation of the ROM matrices and tensors are summarized in Algorithm \ref{alg2}.
\begin{algorithm}[H]
\caption{ROM matrices and tensors calculation of quadratic-bilinear DAEs}\label{alg2}
\begin{algorithmic}[1] 
    \State Solve equation $\vect A_{21}\vect A_{12}\vectg \theta=\vect A_{21}\vect A_{11}\vect w$ for $\vectg\theta$, for $\vect w$ as a column of either $\vect W_p$ or $\vect W_{pq}$.
    
    \State Solve equation $\vect A_{21}\vect A_{12}\vect \Gamma_2\vect e_p=\vect A_{21}(-\vect A_{11}\vect \Gamma_1 +\vect B_1)\vect e_p$ for $\vect \Gamma_2\vect e_p$, the $p$th column for $\vect \Gamma_2$
    
    \State Calculate the vector results of the following matrices or tensors, for $\vect w$ and $\vect w'$ as a column of either $\vect W_p$ or $\vect W_{pq}$, and $\vectg\theta$ and $\vectg\theta'$ as the results by Step 1 of this algorithm, corresponds to $\vect w$ and $\vect w'$, respectively:
    
\leavevmode$\begin{aligned}
&\tensor H\left(\vect w,\vect w'\right)
  =\tensor H_{11}\left(\vect w,\vect w'\right)
  -\tensor H_{12}\left(\vect w,\vectg \theta'\right)\\
&\tensor N\left(\vect w\right) \vect e_p
  =-\tensor H_{11}\left(\vect \Gamma_1 \vect e_p, \vect w\right)
  -\tensor H_{11}\left(\vect w, \vect \Gamma_1 \vect e_p\right)
  +\tensor N_1\left(\vect w\right)\vect e_p
  -\tensor H_{12}\left(\vect w,\vectg \Gamma_2 \vect e_p\right)
  +\tensor H_{12}\left(\vect \Gamma_1 \vect e_p, \vectg \theta\right)
  -\tensor N_2\left(\vectg \theta\right)\vect e_p\\
&\tensor M\left(\vect w\right) \vect e_p
  =-\tensor H_{12}\left(\vect w,
  \vect \Gamma_3 \vect e_p\right)\\
&\tensor D\left(\vect e_p,\vect e_q\right)
  =\tensor H_{11}\left(
  \vect \Gamma_1 \vect e_p,\vect \Gamma_1 \vect e_q\right)
  -\tensor N_1\left(\vect \Gamma_1 \vect e_p\right)\vect e_q 
  +\tensor D_1\left(\vect e_p, \vect e_q\right)
  +\tensor H_{12}\left(\vect \Gamma_1 \vect e_p, \vectg \Gamma_2 \vect e_q\right)
  -\tensor N_2\left(\vectg \Gamma_2 \vect e_p\right)\vect e_q\\
&\tensor K\left(\vect e_p,\vect e_q\right)
  =\tensor H_{12}\left(\vect \Gamma_1 \vect e_p, \vect \Gamma_3 \vect e_q\right)
  -\tensor N_2\left(\vect \Gamma_3 \vect e_q\right)\vect e_p
\end{aligned}$
	\State Calculate the ROM matrices and tensors by the vector results of Step 3, using Equations (\ref{rommatrix1})--(\ref{rommatrix3}).
  	\end{algorithmic}
\end{algorithm}

In the procedure of Algorithm \ref{alg2}, the first two steps can be solved with the native discretization scheme and solution procedure of a CFD solver. Step 3 only requires evaluating the discretized operators using the solutions obtained from Algorithm \ref{alg1}, which can also be calculated by the native discretization scheme of a CFD solver. Finally, Step 4 involves only the computation of several inner products. Thus, the ROM matrices and tensors can be constructed without the direct calculation of the system matrices and tensors.

\subsection{Inner product for ROM matrices calculation of CFD problems}\label{G_FVM}
Step 4 of Algorithm \ref{alg2} involves the calculation of inner product $\langle \cdot,\cdot \rangle_{\vect G}$, where the $\vect G$ must be chosen such that the operator $\vect \Pi$ is self-adjoint with respect to the inner product. In particular, for CFD problems, $\vect A_{12}$ and $\vect A_{21}$ typically correspond to, or are closely related to, the discrete gradient operator $\mathcal G$ and divergence operator $\mathcal D$, respectively. Thus, in most practical cases, the self-adjointness of $\vect \Pi$ under the weighted inner product is generally equivalent to that of the operator $\mathcal G(\mathcal D\mathcal G)^{-1}\mathcal D$ under the corresponding continuous inner product. 
In the finite element method, Ern and Guermond \cite{E2004} showed that the discrete gradient and divergence operators are naturally adjoint under appropriate boundary conditions, as they are derived from the same weak formulation. The corresponding discrete inner products $\langle \cdot,\cdot \rangle_{\vect G}$ are naturally induced by the finite element mass matrices. 
However, in the finite volume method (FVM), the adjoint relation is generally not preserved under the Euclidean inner product. This section presents the construction of the matrix $\vect G$ for CFD problems using FVM. First, we prove the following proposition. 

\begin{prop}\label{adjoint}
Let $\mathcal G$ and $\mathcal D$ denote the discrete gradient and divergence operators obtained by the FVM, associated with the boundary face values $p_f$ and $\vecf v_f$. Define the volume weighted inner products to be $\langle \vecf u,\vecf v\rangle_{\vecf v}:=\sum_KV_K\vecf u_K\cdot\vecf v_K$, where $K$ loops the cell in $\Omega$ and $V_k$ is the volume of the cell $K$. Assume that the Green-Gauss cell-based method is used in the discretization. Furthermore, assume that, for every boundary face $f\in\partial K\cap\partial\Omega$:
\begin{gather}\label{condition}
\left(p_f\vecf v_K+p_K\vecf v_f-p_K\vecf v_K\right)\cdot\vecf n_{Kf}=\vecf 0
\end{gather}
Where cell $K$ is adjacent to face $f$, and $\vecf n_{Kf}$ is the outward unit normal vector to face $f$ with respect to cell $K$. If $\mathcal D\mathcal G$ is invertible, then the discrete operator $\mathcal P:=\mathcal G(\mathcal D\mathcal G)^{-1}\mathcal D$ is self-adjoint with respect to $\langle \cdot,\cdot\rangle_{\vecf v}$.
\end{prop}

\begin{proof}
The discrete gradient and divergence operators by FVM is:
\begin{gather}
	\left(\mathcal G p\right)_K=\frac{1}{V_K}\sum_{f\in \partial K}A_fp_f\vecf n_{Kf},\quad
	\left(\mathcal D \vecf v\right)_K=\frac{1}{V_K}\sum_{f\in \partial K}A_f\vecf v_f\cdot\vecf n_{Kf}
\end{gather}
Where $V_K$ is the volume of the cell $K$, $A_f$ is the area of the face $f$, and $\vecf n_{Kf}$ is the outward unit normal vector to face $f$ with respect to cell $K$. For an interior face $f\in\mathcal F_{\mathrm{int}}$ shared by cells $K$ and $L$, the Green–Gauss cell-based method gives:
\begin{gather}\label{bcd}
	\vecf v_f=\frac{\vecf v_K+\vecf v_L}{2},\quad
	p_f=\frac{p_K+p_L}{2}	
\end{gather}
Define a volume weighted inner products of $p$ to be $\langle p,q \rangle_{p} :=\sum_KV_Kp_Kq_K$, then:
\begin{gather}
\left\langle \vecf v,\mathcal Gp\right\rangle_{\vecf v}=\sum_K\sum_{f\in \partial K}A_fp_f\vecf v_K\cdot\vecf n_{Kf}
=\sum_{f\in \mathcal F_{\mathrm{int}}}A_fp_f\left(\vecf v_K\cdot\vecf n_{Kf}+\vecf v_L\cdot\vecf n_{Lf}\right)+\sum_{f\in \partial \Omega}A_fp_f\vecf v_K\cdot\vecf n_{Kf}	\\
\left\langle p,\mathcal D \vecf v\right\rangle_{p}=\sum_K\sum_{f\in \partial K}A_fp_K\vecf v_f\cdot\vecf n_{Kf}=\sum_{f\in \mathcal F_{\mathrm{int}}}A_f\left(p_K\vecf v_f\cdot\vecf n_{Kf}+p_L\vecf v_f\cdot\vecf n_{Lf}\right)+\sum_{f\in \partial \Omega}A_fp_K\vecf v_f\cdot\vecf n_{Kf}
\end{gather}
where K and L are the two cells sharing the interior face $f$ in the second step of two equations. $\vecf n_{Kf}$ and $\vecf n_{Lf}$ is unit normal vector of one face $f$, but in different direction, thus $\vecf n_{Kf}=-\vecf n_{Lf}$. Then:
\begin{gather}
	\left\langle \vecf v,\mathcal Gp\right\rangle_{\vecf v}+\left\langle p,\mathcal D \vecf v\right\rangle_{p}=\sum_{f\in \mathcal F_{\mathrm{int}}}A_f\left(p_f\left(\vecf v_K-\vecf v_L\right)+\left(p_K-p_L\right)\vecf v_f\right)\cdot\vecf n_{Kf}+\sum_{f\in \partial \Omega}A_f\left(p_f\vecf v_K+p_K\vecf v_f\right)\cdot\vecf n_{Kf}	
\end{gather}
Using Equation (\ref{bcd}) gives:
\begin{gather}
\begin{split}
	&\left\langle \vecf v,\mathcal Gp\right\rangle_{\vecf v}+\left\langle p,\mathcal D \vecf v\right\rangle_{p}\\
	=&\sum_{f\in \mathcal F_{\mathrm{int}}}A_f\left(p_K\vecf v_K-p_L\vecf v_L\right)\cdot\vecf n_{Kf}+\sum_{f\in \partial \Omega}A_fp_K\vecf v_K\cdot\vecf n_{Kf}+\sum_{f\in \partial \Omega}A_f\left(p_f\vecf v_K+p_K\vecf v_f-p_K\vecf v_K\right)\cdot\vecf n_{Kf}
\end{split}	
\end{gather}
The third term vanishes by assumption and the remaining two terms of the equation can be written as:
\begin{gather}
 	\sum_K\sum_{f\in \partial K}A_fp_K\vecf v_K\cdot\vecf n_{Kf}=\sum_Kp_K\vecf v_K\cdot\sum_{f\in \partial K}A_f\vecf n_{Kf}
\end{gather}
By Divergence theorem:
\begin{gather}
 	\sum_{f\in \partial K}A_f\vecf n_{Kf}=\int_{\partial K}\vect ndS=\int_{K}\vect \nabla 1 dV=0
\end{gather}
Thus, $\langle \vecf v,\mathcal Gp\rangle_{\vecf v}=-\langle p,\mathcal D \vecf v\rangle_{p}$. Then,
\begin{gather}
	\left\langle p,\mathcal D\mathcal Gq\right\rangle_{p}=-\left\langle \mathcal Gp,\mathcal Gq\right\rangle_{\vecf v}=\left\langle \mathcal D\mathcal Gp,q\right\rangle_{p}
\end{gather}
Thus, $\mathcal D\mathcal G$ is self-adjoint with respect to $\langle\cdot,\cdot\rangle_p$. Since $\mathcal D\mathcal G$ is invertible by assumption, $(\mathcal D\mathcal G)^{-1}$ is also self-adjoint. Therefore:
\begin{gather}
	\left\langle \vecf u,\mathcal P\vecf v\right\rangle_{\vecf v}=-\left\langle \mathcal D\vecf u,\left(\mathcal D\mathcal G\right)^{-1}\mathcal Dv\right\rangle_{p}=-\left\langle \left(\mathcal D\mathcal G\right)^{-1}\mathcal D\vecf u,\mathcal Dv\right\rangle_{p}=\left\langle \mathcal P\vecf u,\vecf v\right\rangle_{\vecf v}
\end{gather}
Which proves the proposition.
\end{proof}

When $p_f-p_K=0$ and $\vecf u_f\cdot\vecf n_{Kf} =0$, Equation (\ref{condition}) is satisfied. This condition correspond to: 
\begin{gather}
\left.\vecf{u}\cdot \vecf n\right|_{\Gamma}=0, \quad
\left.\frac{\partial p}{\partial \vecf{n}}\right|_{\Gamma}=0
\end{gather}
Likewise, when $p_f=0$ and $(\vecf u_f-\vecf u_K)\cdot\vecf n_{Kf}=0$, Equation (\ref{condition}) is also satisfied. This condition correspond to: 
\begin{gather}
\left.\frac{\partial \vecf{u}}{\partial \vecf{n}}\cdot \vecf n\right|_{\Gamma}=0, \quad
\left.p\right|_{\Gamma}=0
\end{gather}

If, on each boundary, either of these two sets of condition is satisfied, then, according to Proposition \ref{adjoint}, the weighting matrix $\vect G$ should be chosen to be a diagonal matrix, with each diagonal entry equal to the volume of the corresponding cell, so that the $\vect \Pi$ is self-adjoint with respect to the inner product $\langle \cdot,\cdot \rangle_{\vect G}$. This inner product should be used in Step 4 of Algorithm \ref{alg2}.

\section{Numerical Experiment}\label{NumEx}

In this section, we examine the performance of the proposed ROMs for quadratic-bilinear descriptor systems using two examples of CFD problems. The transient response of the ROM is compared with the FOM by changing the input of the system from a prescribed operating point to new input values. 

\subsection{T-junction flow problem}

Junctions are widely used to combine or separate fluid flows and are encountered in a broad range of engineering applications, from gas transport in pipeline systems to microfluidic reactors \cite{S2021}. A test case of a T-junction with two inlets and one outlet is modeled in this section. The flow is assumed to be two-dimensional and incompressible. In this case, the conservation equations are:
\begin{subequations}
\begin{gather}
\vecf{\nabla} \cdot \vecf{v} =0 \label{e1C1}\\
\frac{\partial\vecf{v}}{\partial t}
	+\vecf{v}\cdot\vecf{\nabla}\vecf{v}
	=-\vecf{\nabla}p+\frac{1}{\mathrm{Re}}\Delta\vecf{v} \label{e1V1}
\end{gather}
\end{subequations}
where Re is the Reynolds number, which is set to $100$ in this case. The left and top inlet boundaries are $\Gamma_{i1}$ and $\Gamma_{i2}$, respectively. The wall boundary is $\Gamma_{w}$ and the outlet boundary is $\Gamma_{o}$. Thus, the boundary condition is:
\begin{gather}
\left.\vecf{v}\right|_{\Gamma_{i1}}=\vecf{v}_{i1}, \quad
\left.\vecf{v}\right|_{\Gamma_{i2}}=\vecf{v}_{i2}, \quad
\left.\vecf{v}\right|_{\Gamma_{w}}=0, \quad
\left.\frac{\partial\vecf{v}}{\partial \vecf{n}}\right|_{\Gamma_{o}}=0
\end{gather}
The ROM is constructed about the operating point corresponding to the steady-state solution $\vecf{v}_s$ and $p_s$, obtained with the boundary conditions $\vecf{v}_{i1}=[1,0]^T$ and $\vecf{v}_{i2}=[0,-1]^T$. Figure \ref{fig:E1f1} shows the scheme of the T-junction flow problem and the velocity magnitude at the operating point. Changes in the magnitudes of the inlet velocities at the two inlets are treated as the system inputs $\vect u = [u_1,u_2]^T$, while the flow directions remain unchanged and are always normal to the inlet boundaries. The solution is decomposed into the steady-state and fluctuating parts: $\vecf{v}=\vecf{v}_s+\vecf{v}_\delta$ and $p=p_s+p_\delta$. The equations become:
\begin{subequations}
\begin{gather}
\vecf{\nabla} \cdot \vecf{v}_\delta =0 \label{e1C2}\\
\frac{\partial\vecf{v}_\delta}{\partial t}=
	-\vecf{\nabla}p_\delta
	-\vecf{v}_s\cdot\vecf{\nabla}\vecf{v}_\delta
	-\vecf{v}_\delta\cdot\vecf{\nabla}\vecf{v}_s
	-\vecf{v}_\delta\cdot\vecf{\nabla}\vecf{v}_\delta
	+\frac{1}{\mathrm{Re}}\Delta\vecf{v}_\delta \label{e1V2}
\end{gather}
\end{subequations}
With boundary conditions:
\begin{gather}
\left.\vecf{v}_\delta\right|_{\Gamma_{i1}}=\left[u_1,0\right]^T, \quad
\left.\vecf{v}_\delta\right|_{\Gamma_{i2}}=\left[0,-u_2\right]^T, \quad
\left.\vecf{v}_\delta\right|_{\Gamma_{w}}=0, \quad
\left.\frac{\partial\vecf{v}_\delta}{\partial \vecf{n}}\right|_{\Gamma_{o}}=0
\end{gather}

\begin{figure}[H]
\begin{center}
\includegraphics[width=0.75\textwidth]{ 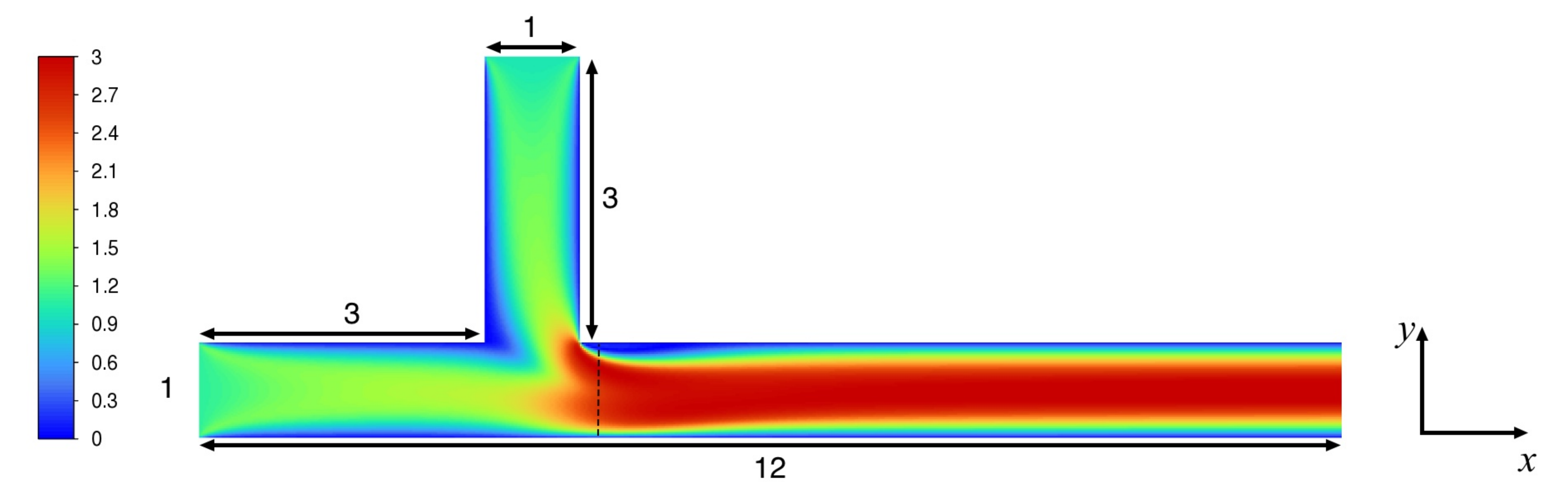}
\caption{Scheme of T-junction flow problem and the velocity magnitude at the operating point.}
\label{fig:E1f1}
\end{center}
\end{figure}

The system is discretized in ANSYS Fluent using the FVM, with the Green-Gauss cell-based method and the second-order central differencing scheme. 
%In the discretization of the deviation equations, $\vecf{v}_s\cdot\vecf{\nabla}\vecf{v}_\delta$ is treated as a convection term, with the steady-state velocity $\vecf{v}_s$ serving as the known advecting velocity.
The momentum equation is discretized into differential equations, and the continuity equation and boundary conditions become the singular part and non-singular part of algebraic equations, respectively. The algebraic equations are obviously linear, and the only nonlinear term in the differential equations is $\vecf{v}_\delta\cdot\vecf{\nabla}\vecf{v}_\delta$, which contains the quadratic term by differential and input variables at the cell adjacent to the boundary, and the quadratic term by pure differential variable at inner cell. The system of $\vect x=[\vect x_1,\vect x_2]^T$, where $\vect x_1$ and $\vect x_2$ correspond to the discretized $\vecf{v}_\delta$ and $p_\delta$ at each cell, respectively, can be written as:
\begin{subequations}
\begin{gather}
 \dot{\vect x_1}=\vect A_{11}\vect x_1+\vect A_{12}\vect x_2
  +\vect B_1\vect u
  +\tensor H_{11}\left(\vect x_1, \vect x_1\right)
  +\tensor N_1\left(\vect x_1\right)\vect u\\
 \vect 0=\vect A_{21}\vect x_1+\vect B_2\vect u
\end{gather}
\end{subequations}
where the $\vect B_1$ and $\vect B_2$ term corresponds to the source applied within the cell adjacent to the inlet boundaries by the manipulate variables, $\tensor H_{11}(\vect x_1, \vect x_1)$ corresponds to the convection of $\vecf{v}_\delta$ by itself, and $\tensor N_1(\vect x_1)\vect u$ corresponds to the convection term by the inlet boundary velocity at the cell adjacent to the inlet boundaries. 

After eliminating the boundary variables and incorporating their contributions into the source terms, $\vect A_{21}$ represents the discrete divergence operator $\mathcal D$ with no-penetration boundary conditions on the walls and inlets and homogeneous Neumann boundary conditions at the outlet, while $\vect A_{12}$ represents the negative discrete gradient operator $-\mathcal G$ with homogeneous Dirichlet boundary conditions at the outlet and homogeneous Neumann boundary conditions on the walls and inlets. Consequently, $\vect A_{21}\vect A_{12}$ corresponds to the discrete negative Laplacian operator with homogeneous Dirichlet boundary condition at the outlet and is therefore non-singular. Hence, Assumption \ref{as2} is satisfied. Furthermore, according to Proposition \ref{adjoint}, the weighting matrix $\vect G$ should be chosen as a diagonal matrix whose diagonal entries are the volumes of the corresponding cells, so that $\vect\Pi$ is self-adjoint with respect to the weighted inner product $\langle\cdot,\cdot\rangle_{\vect G}$.

$\vect \Gamma_3$ can be solved by a Laplace's equation, and $\vect \Gamma_1$ can be calculated by $\vect \Gamma_3$. The Krylov subspaces can be constructed, and the ROM matrices can be calculated based on the derivation in this study. In this case, four interpolation points, $0$, $0.01$, $0.1$, and $1$, were used to construct the ROM. Since the terms corresponding to Step 4 of Algorithm \ref{alg1} are absent, the associated Krylov subspace is empty, resulting in a 20th-order ROM. These interpolation points were selected based on the approximate time scales observed in the system response. 
The ROM was constructed based on the FOM solution on a uniform mesh with an element size of 0.02, and its transient responses were compared with the numerical simulation results obtained from the FOM using the same mesh. The mesh consists of 37,500 grid cells, resulting in 112,500 degrees of freedom corresponding to two velocity components and one pressure component in each cell.
Figure \ref{fig:E1f2} shows the input variation used to test the ROM transient response. In the first case, the inlet velocity variation was applied only to the top boundary condition, $\vecf{v}_{i2}$. In the second case, the same inlet velocity variation was applied simultaneously to the left and top boundary conditions, $\vecf{v}_{i1}$ and $\vecf{v}_{i2}$. To demonstrate the effect of the cross terms, the ROM predictions were also compared with a ROM in which the cross terms were neglected. In that case, the response was obtained by simply adding the individual responses to the variation of $\vecf{v}_{i1}$ and $\vecf{v}_{i2}$.

\begin{figure}[H]
\begin{center}
\includegraphics[width=0.31\textwidth]{ 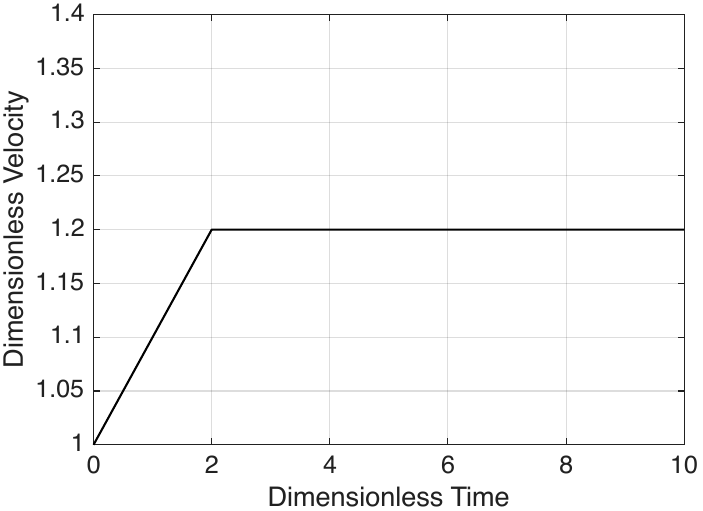}
\caption{Inlet velocity variation applied to the systems.}
\label{fig:E1f2}
\end{center}
\end{figure}

Figure \ref{fig:E1f3} shows the relative $L^2$ error as a function of time of the three cases, which is defined as the following: 
\begin{gather}
	e\left(t\right)=\frac{\left\|\hat{\vecf v}_\delta\left(t\right)-\vecf v_\delta\left(t\right)\right\|_{L^2}}{\left\|\vecf v_\delta\left(t\right)\right\|_{L^2}}
\end{gather}
where $\hat{\vecf v}_\delta$ denotes the $\vecf v_\delta$ predicted by the ROM. For the first two cases, the relative $L^2$ error is larger at the beginning of the simulation and gradually decreases with time. The relatively large initial error is because the response contains high-frequency dynamics at the beginning due to the input variation, which is difficult for the ROM to capture. And the gradual decrease of the error shows that the ROM accurately captures the dominant low-frequency dynamics. For the third case, which uses the same input variation as the second case but neglects the cross terms, the error is relatively higher throughout the simulation. This demonstrates that the Krylov subspaces associated with the cross terms successfully capture the second-order coupling effects between the two inputs.

\begin{figure}[H]
\begin{center}
\includegraphics[width=0.36\textwidth]{ 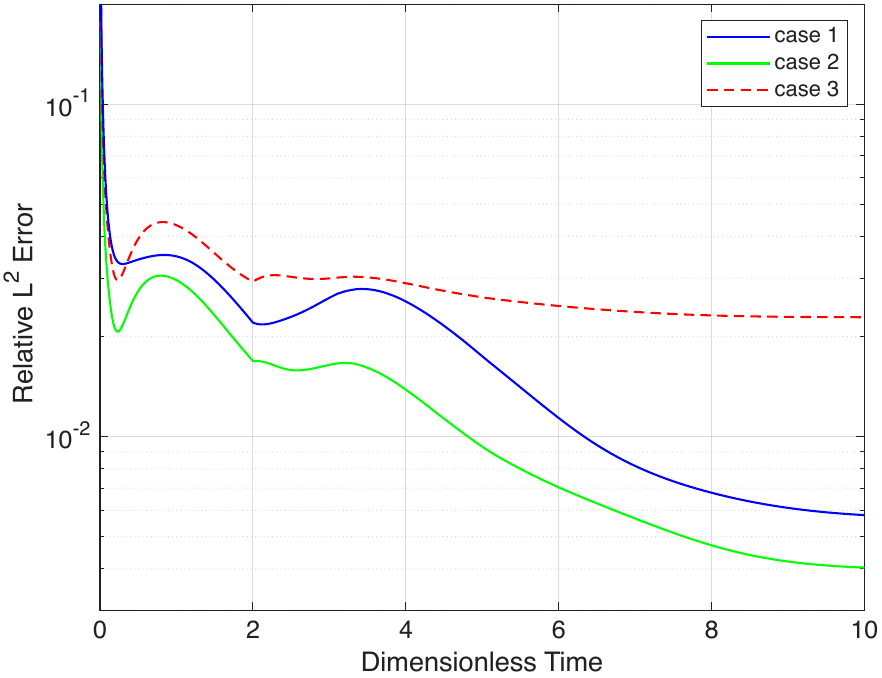}
\caption{Relative $L^2$ error as a function of time for three cases: (1) variation in $\vecf{v}_{i2}$ only; (2) simultaneous variations in $\vecf{v}_{i1}$ and $\vecf{v}_{i2}$; and (3) simultaneous variations in $\vecf{v}_{i1}$ and $\vecf{v}_{i2}$ without the cross terms.}
\label{fig:E1f3}
\end{center}
\end{figure}

Furthermore, $\vecf v_\delta$ and $\hat{\vecf{v}}_\delta$ at several points on the dashed line in Figure \ref{fig:E1f1} as a function of time is plotted for the three cases. This line is selected because it is in the region where the flow from two inputs meets, where there would be significant coupling effects. Six points, denoted by $a$ to $f$, are chosen on the dashed line. Their distances from the upper wall are 0.1, 0.2, 0.3, 0.5, 0.7, and 0.9, respectively. As shown in Figure \ref{fig:E1f4}, the $\hat{\vecf{v}}_\delta$ agrees well with $\vecf v_\delta$ for the first two cases throughout the transient process, and both converge to nearly identical steady-state solutions. In contrast, for the third case, the two models do not even converge to similar steady-state solutions, especially for points c and e. These results demonstrate the importance of the Krylov subspaces associated with the cross terms and show that the proposed ROM successfully captures the second-order coupling effects between the two inputs.

\begin{figure}[H]
\begin{center}
\includegraphics[width=0.7\textwidth]{ 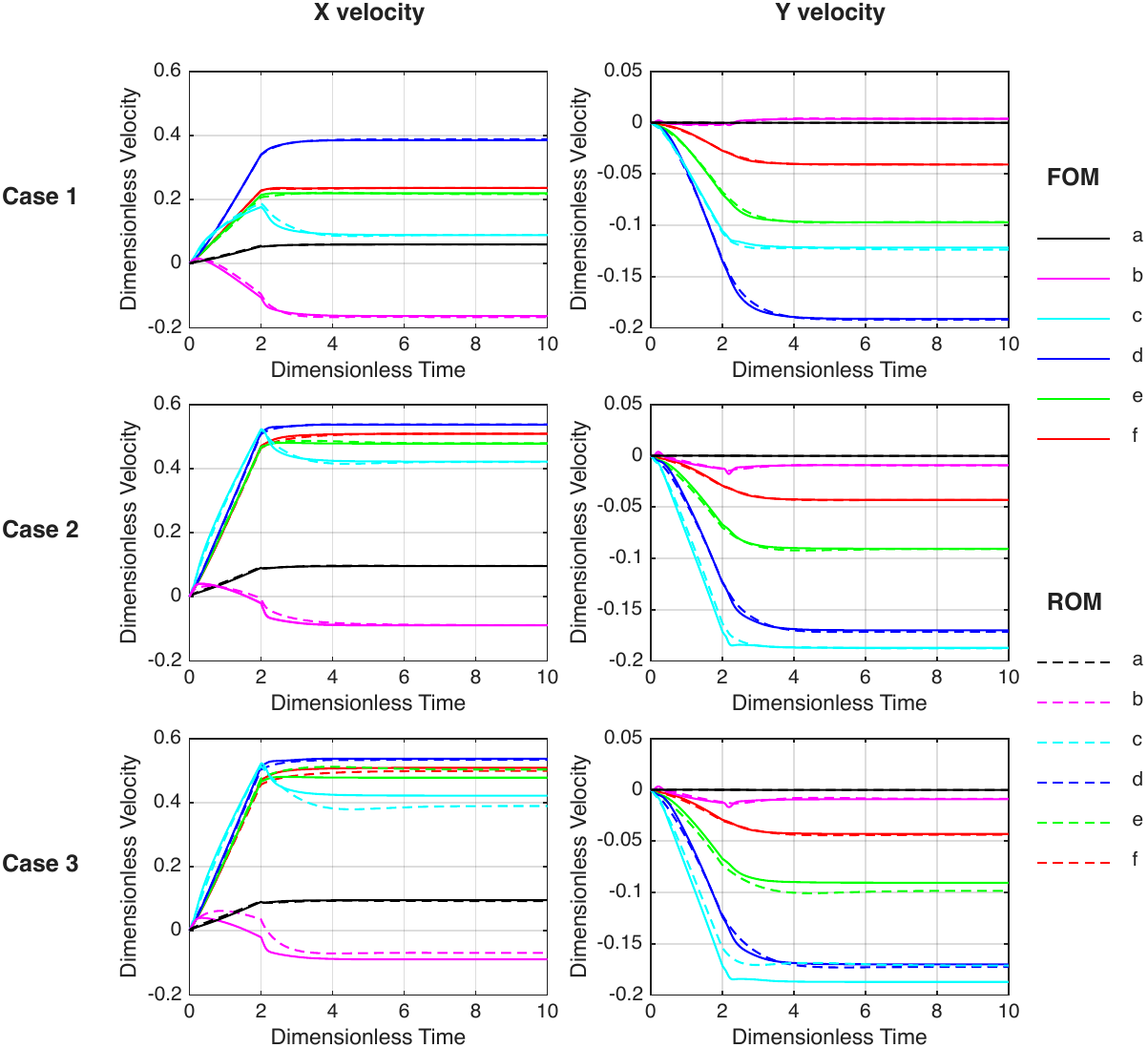}
\caption{Comparison of $\vecf v_\delta$ and $\hat{\vecf{v}}_\delta$ at the points on the dashed line in Figure \ref{fig:E1f1} as a function of time for three cases: (1) variation in $\vecf{v}_{i2}$ only; (2) simultaneous variations in $\vecf{v}_{i1}$ and $\vecf{v}_{i2}$; and (3) simultaneous variations in $\vecf{v}_{i1}$ and $\vecf{v}_{i2}$ without the cross terms.}
\label{fig:E1f4}
\end{center}
\end{figure}

Figure \ref{fig:E1f5} compares the magnitude of ${\vecf{v}_\delta}$ and $\hat{\vecf{v}}_\delta$, together with the corresponding absolute error $|\hat{\vecf{v}}_\delta-{\vecf v_\delta}|$, for the case with simultaneous variations in $\vecf{v}_{i1}$ and $\vecf{v}_{i2}$. The FOM and ROM predictions are visually indistinguishable, and the absolute error contour indicates that the maximum absolute error is approximately $1\%$ of the maximum magnitude of ${\vecf{v}_\delta}$. The largest errors occur near the upper wall in the flow interaction region, where the coupling between the two inlet flows is strongest, and the local flow gradients are the largest.

\begin{figure}[H]
\begin{center}
\includegraphics[width=0.72\textwidth]{ 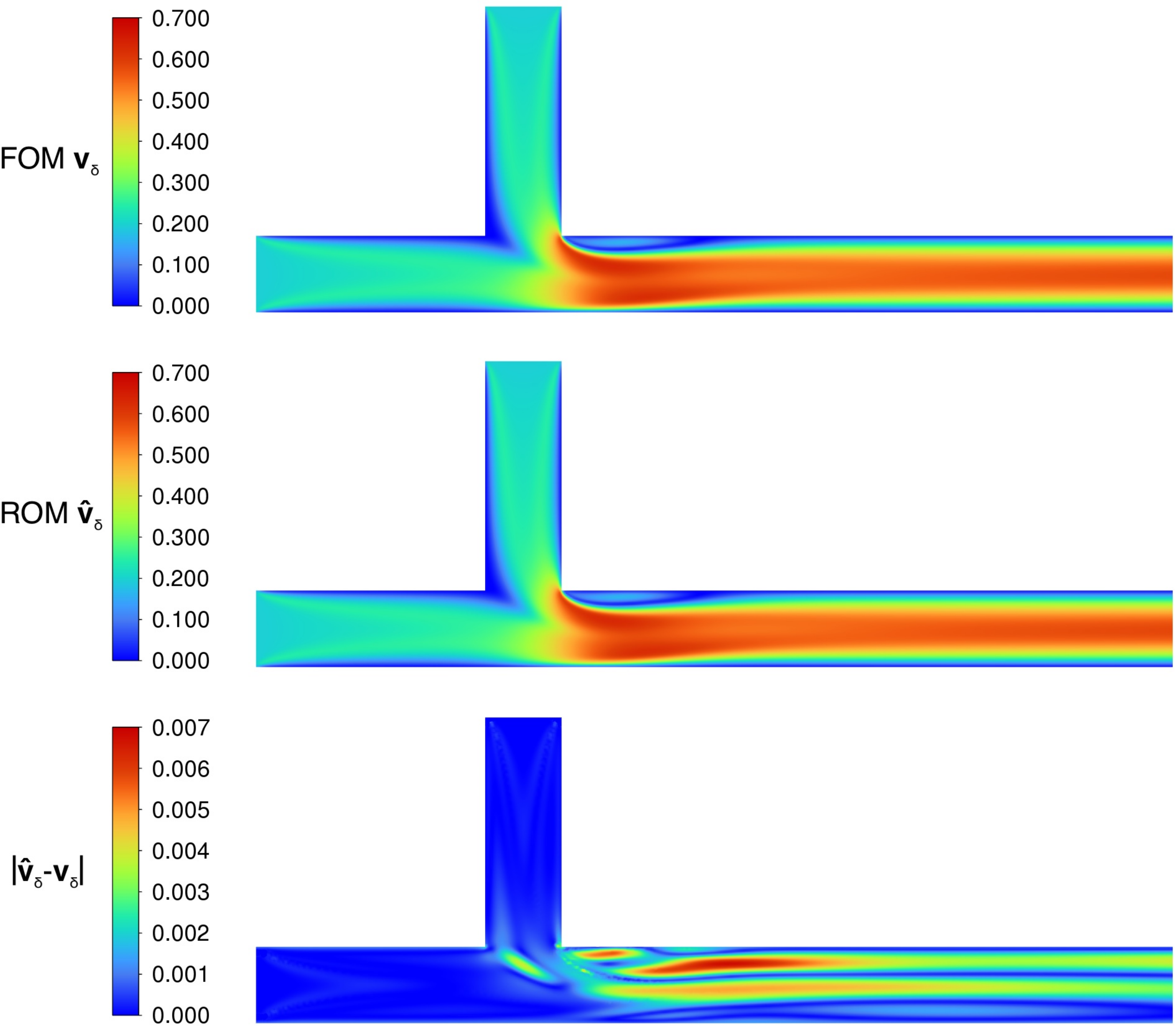}
\caption{Comparison of the magnitude of ${\vecf{v}_\delta}$ and $\hat{\vecf{v}}_\delta$, together with the corresponding absolute error $|\hat{\vecf{v}}_\delta-{\vecf v_\delta}|$, for the case with simultaneous variations in $\vecf{v}_{i1}$ and $\vecf{v}_{i2}$.}\label{fig:E1f5}
\end{center}
\end{figure}

\subsection{Fluid flow problem coupled with heat transfer}

For the second example, we consider a fluid flow problem coupled with heat transfer. The velocity field transports thermal energy through forced convection, while the temperature field affects the density, which is coupled back to the flow field through the governing equations. Figure \ref{fig:E2f1} shows the schematic diagram of the problem, which includes a volumetric source term $f$ in the $y$-momentum equation and a volumetric energy source $q$. The example is inspired by a conveyor convection oven and can be readily extended to other coupled fluid flow and heat transfer systems, including HVAC systems and heat exchangers. The left and right channels represent the entrance and exit of a conveyor convection oven, respectively. In this example, they are modeled as an air inlet and an air outlet. The flow is assumed to be two-dimensional and compressible. The computational domain is chosen to be much smaller than that of a real oven because turbulence is not considered in the present study. A smaller characteristic length is needed to maintain a low Reynolds number to ensure laminar flow.
In this case, the conservation equations are:
\begin{subequations}\label{VCT1}
\begin{gather}
\frac{\partial \rho}{\partial t}
	+\vecf{\nabla} \cdot \left(\rho \vecf{v}\right)
	=0 \label{C1}\\
\frac{\partial\rho\vecf{v}}{\partial t}
	+\vecf{\nabla}\cdot\left(\rho\vecf{v}\vecf{v}\right)
	=-\vecf{\nabla}p+\mu\Delta\vecf{v}
	+\frac{\mu}{3}\vecf\nabla\left(\vecf\nabla\cdot\vecf{v}\right)
	+\rho\vecf{g}+\vecf{f} \label{V1}\\
\frac{\partial\rho C_p T}{\partial t}
	+\vecf{\nabla}\cdot\left(\rho C_p \vecf{v}T\right)
	=k \Delta T+q\label{T1}
\end{gather}
\end{subequations}

\begin{figure}[H]
\begin{center}
\includegraphics[width=0.6\textwidth]{ 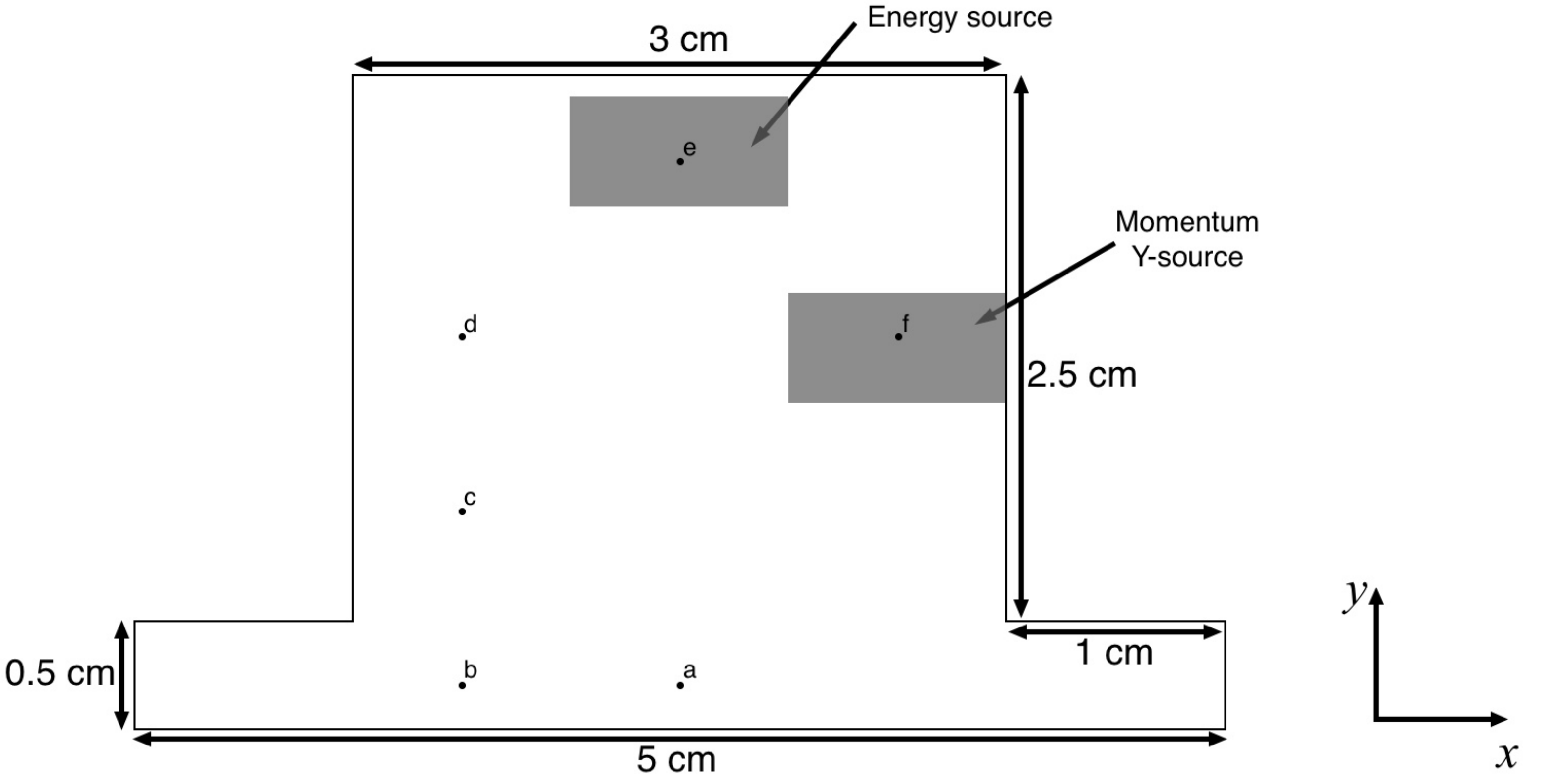}
\caption{Scheme of fluid flow problem coupled with heat transfer.}
\label{fig:E2f1}
\end{center}
\end{figure}

We assume that, because of the small pressure variation, the incompressible ideal-gas approximation can be used to simplify the conservation equations:
\begin{gather}
	\rho T=\text{const}=\rho_0T_0 \label{CE}
\end{gather}
Subtracting Equation (\ref{C1}) from Equation (\ref{V1}), dividing by $\rho$, and using Equation (\ref{CE}) gives:
\begin{subequations}\label{VCT2}
\begin{gather}
\frac{\partial\vecf{v}}{\partial t}
	+\vecf{v}\cdot\vecf{\nabla}\vecf{v}
	=\frac{T}{\rho_0 T_0}\left(-\vecf{\nabla}p
	+\mu\Delta\vecf{v}
	+\frac{\mu}{3}\vecf\nabla\left(\vecf\nabla\cdot\vecf{v}\right)+\vecf{f}\right)
	+\vecf{g} \label{V2}
\end{gather}
Similarly, subtracting Equation (\ref{C1}) from Equation (\ref{T1}) gives:
\begin{gather}
\frac{\partial T}{\partial t}
	+\vecf{v}\cdot\vecf{\nabla}T=
	\frac{T}{\rho_0 T_0 C_p}\left(k\Delta T+q \right) \label{T2}
\end{gather}
Since $\rho T$ is constant, Equation (\ref{T1}) becomes:
\begin{gather}
\vecf{\nabla}\cdot \vecf{v}
	-\frac{k}{\rho_0 T_0 C_p}\Delta T
	=\frac{q}{\rho_0 T_0 C_p} \label{C2}
\end{gather}
\end{subequations}

The ROM is constructed about the operating point corresponding to the steady-state solution with $\vecf{f}=\vecf{f}_s$ and ${q}={q}_s$. The operating-point values of $\vecf{f}_s$ and $q_s$ are $4.5~\mathrm{N/m^3}$ and $80~\mathrm{kW/m^3}$, respectively. The corresponding steady-state solution is shown in Figure \ref{fig:E2fss}. The deviation of $\vecf{f}$ and ${q}$, denoted as $\vecf{f}_\delta$ and ${q}_\delta$, are regarded as the input to the system. Decomposing $\vecf{v}=\vecf{v}_s+\vecf{v}_\delta$, $T=T_s+T_\delta$, and $p=p_s+p_\delta$, Equations (\ref{VCT2}) becomes, as equations of $(\vecf{v}_\delta, T_\delta, p_\delta)$:
\begin{subequations}\label{VCT3}
\begin{gather}
\vecf{\nabla}\cdot \vecf{v}_\delta
	-\frac{k}{\rho_0 T_0 C_p}\Delta T_\delta
	=\frac{q_\delta}{\rho_0 T_0 C_p} \label{C3}\\
\begin{split}
\frac{\partial\vecf{v}_\delta}{\partial t}
	+\vecf{v}_s\cdot\vecf{\nabla}\vecf{v}_\delta
	+\vecf{v}_\delta\cdot\vecf{\nabla}\vecf{v}_s&=
	\frac{T_s}{\rho_0 T_0}\left(-\vecf{\nabla}p_\delta
	+\mu\Delta\vecf{v}_\delta+\frac{\mu}{3}\vecf\nabla\left(\vecf\nabla\cdot\vecf{v}_\delta\right)
	+\vecf{f}_\delta\right)\\
	&+\frac{T_\delta}{\rho_0 T_0}\left(-\vecf{\nabla}p_s
	+\mu\Delta\vecf{v}_s+\frac{\mu}{3}\vecf\nabla\left(\vecf\nabla\cdot\vecf{v}_s\right)
	+\vecf{f}_s\right) 
	+ \vecf{\mathcal F_v}\left(\vecf{v}_\delta, T_\delta, p_\delta, \vecf{f}_\delta\right)\label{V3}
\end{split}\\
\frac{\partial T_\delta}{\partial t}
	+\vecf{v}_s\cdot\vecf{\nabla}T_\delta
	+\vecf{v}_\delta\cdot\vecf{\nabla}T_s=
	\frac{T_s}{\rho_0 T_0 C_p}\left(k\Delta T_\delta+q_\delta \right)
	+\frac{T_\delta}{\rho_0 T_0 C_p}\left(k\Delta T_s+q_s \right) 
	+{\mathcal F}_T\left(\vecf{v}_\delta,T_\delta, q_\delta\right) \label{T3}  
\end{gather}
\end{subequations}
Together with the following quadratic terms toward Equation (\ref{V3}) and (\ref{T3}), respectively:
\begin{subequations} \label{SUT}
\begin{gather}
\vecf{\mathcal F}_{\vecf{v}}\left(\vecf{v}_\delta,T_\delta,p_\delta,\vecf{f}_\delta\right)=
	-\vecf{v}_\delta\cdot\vecf{\nabla}\vecf{v}_\delta
	+\mu\frac{T_\delta}{\rho_0 T_0}\Delta\vecf{v}_\delta
	+\frac{\mu}{3}\frac{T_\delta}{\rho_0 T_0}\vecf\nabla\left(\vecf\nabla\cdot\vecf{v}_\delta\right)
	-\frac{T_\delta}{\rho_0 T_0}\vecf{\nabla}p_\delta
	+\frac{T_\delta}{\rho_0 T_0}\vecf{f}_\delta \label{Su}\\
{\mathcal F}_T\left(\vecf{v}_\delta,T_\delta, q_\delta\right)=
	-\vecf{v}_\delta\cdot\vecf{\nabla}T_\delta
	+\frac{k}{\rho_0 T_0 C_p}T_\delta\Delta T_\delta
	+\frac{T_\delta}{\rho_0 T_0 C_p}q_\delta \label{ST}
\end{gather}
\end{subequations}

\begin{figure}[H]
\begin{center}
\includegraphics[width=0.9\textwidth]{ 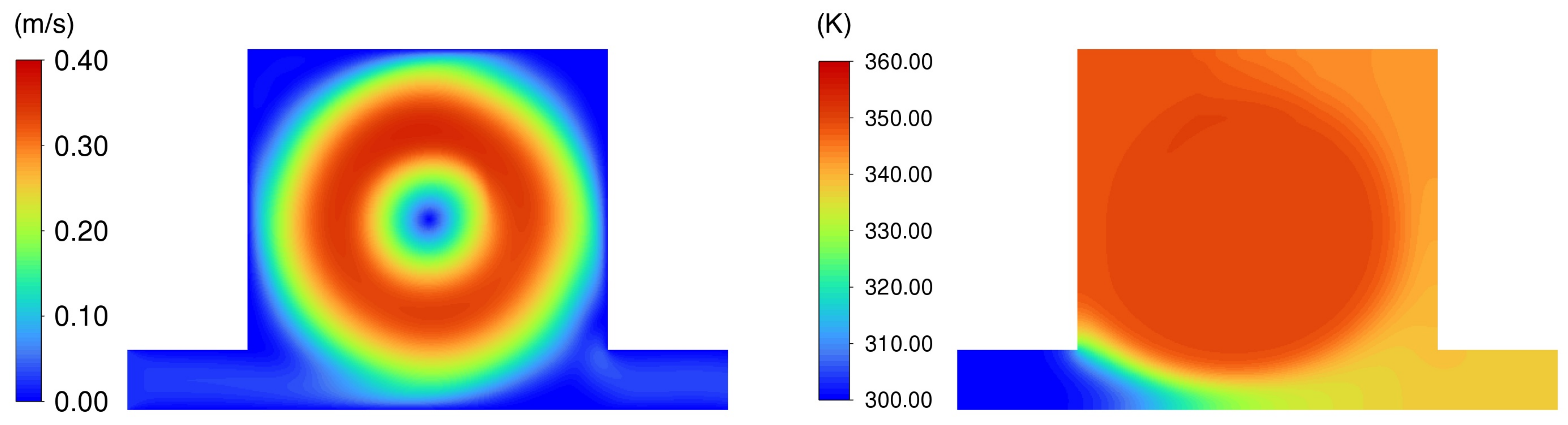}
\caption{Steady-state solution of velocity magnitude and temperature at the operating-point.}
\label{fig:E2fss}
\end{center}
\end{figure}

The deviation system is discretized using the FVM, with the Green-Gauss cell-based method and the second-order central differencing scheme. The momentum and energy equations are discretized into differential equations, while the continuity equation forms the singular algebraic part of the descriptor system. Since $(\vecf{v}, T, p)$ and $(\vecf{v}_s, T_s, p_s)$ satisfy the same boundary conditions, the deviation variables $(\vecf{v}_\delta, T_\delta, p_\delta)$ satisfy homogeneous boundary conditions, which is, either their values or their normal gradients vanish on the corresponding boundaries. Consequently, the non-singular algebraic equations from the boundary conditions vanish. The discretized system can be written in the form of Equation (\ref{DAE}):
\begin{subequations}\label{DCVT}
\begin{gather}
 \dot{\vect x_1}=\vect A_{11}\vect x_1+\vect A_{12}\vect x_2
  +\vect B_1\vect u
  +\tensor H_{11}\left(\vect x_1, \vect x_1\right)
  +\tensor H_{12}\left(\vect x_1, \vect x_2\right)
  +\tensor N_1\left(\vect x_1\right)\vect u\\
 \vect 0=\vect A_{21}\vect x_1+\vect B_2\vect u
\end{gather}
\end{subequations}
where $\vect A_{12}$ and $\vect A_{21}$ are the finite volume discretization of the differential operators: 
\begin{gather*}
 \left[-\frac{1}{\rho_s}\vecf{\nabla}; 0\right],\quad
 \left[\vecf{\nabla}\cdot,-\frac{k}{\rho_0 C_p T_0}\Delta\right] 
\end{gather*}
Consequently, $\vect A_{21}\vect A_{12}$ is a discrete counterpart of the scaled negative Laplacian operator $-\Delta/{\rho_s}$, with homogeneous Dirichlet boundary conditions at the outlet and the corresponding homogeneous Neumann boundary conditions on the remaining boundaries. Thus, it is non-singular, and therefore Assumption \ref{as2} is satisfied. Furthermore, $\vect \Pi$ is the discrete counterpart of $(\vecf I-\vecf{\nabla}\Delta^{-1}\vecf{\nabla}\cdot)\oplus\vecf0$. By Proposition \ref{adjoint}, $\vect \Pi$ is self-adjoint with respect to the weighted inner product $\vect G$, where $\vect G$ is the diagonal matrix whose entries are the corresponding cell volumes.
 
$\vect \Gamma_2$ and $\vect \Gamma_3$ can be obtained by solving the corresponding Poisson equations, and $\vect \Gamma_1$ can be computed by $\vect \Gamma_3$. The Krylov subspaces are then constructed, and the ROM matrices can be calculated following the derivation in this study. 
In the numerical implementation, the equations in Steps 2--4 in Algorithm \ref{alg1} are not directly solved. Instead, based on the analysis in this section, Equation~(\ref{VCT3}), with the second-order terms ${\mathcal F}_T$ and $\vecf{\mathcal F}_{\vecf{v}}$ removed, can be written as: 
\begin{subequations}
\begin{gather}
\frac{\partial \rho}{\partial t}
	+\vecf{\nabla} \cdot \left(\rho \vecf{v}\right)
	=\frac{\rho}{T} {\mathcal F}_T\left(\vecf{v}_\delta,T_\delta, q_\delta\right)\\
\frac{\partial\rho\vecf{v}}{\partial t}
	+\vecf{\nabla}\cdot\left(\rho\vecf{v}\vecf{v}\right)
	=-\vecf{\nabla}p+\mu\Delta\vecf{v}
	+\frac{\mu}{3}\vecf\nabla\left(\vecf\nabla\cdot\vecf{v}\right)
	+\rho\vecf{g}+\vecf{f}
	-\rho \vecf{\mathcal F}_{\vecf{v}}\left(\vecf{v}_\delta,T_\delta,p_\delta,\vecf{f}_\delta\right)\\
\frac{\partial\rho C_p T}{\partial t}
	+\vecf{\nabla}\cdot\left(\rho C_p \vecf{v}T\right)
	=k \Delta T+q
\end{gather}
\end{subequations}
The equations in Steps 2--4 of Algorithm \ref{alg1} can likewise be recast into these forms, which are naturally suited to finite-volume discretization. These equations are solved in ANSYS Fluent using the Green-Gauss cell-based method and the second-order central differencing scheme.
Since the underlying governing equations are equivalent at the continuous level and are discretized using second-order schemes, the resulting linearized discrete equations differ from the linear part of Equation~(\ref{DCVT}) by $O(h^2)$, where $h$ denotes the mesh size. Thus, the solutions of these equations can be used to construct the Krylov subspaces when grid independence is achieved.

The solutions used to construct the Krylov subspaces were obtained on a uniform mesh with an element size of $0.1~\mathrm{mm}$, and the transient responses of the resulting ROM were compared with numerical solutions of Equation~(\ref{VCT1}) computed on the same mesh.
The mesh consists of $100,000$ grid cells, resulting in $400,000$ degrees of freedom corresponding to two velocity components, one temperature component and one pressure component in each cell.
In this example, four interpolation points, $0$, $0.01$, $0.1$, and $1$, were used to construct the ROM, resulting in a 28th-order ROM. Compared with the previous example, the higher ROM order in this example is due to the additional Krylov subspace associated with the time derivative of the energy source. Figure \ref{fig:E2f2} shows the input variations used to test the ROM transient response. Three cases were considered in this example: (1) variation of the momentum source, $\vecf{f}_\delta$, only; (2) variation of the energy source, $q_\delta$, only; and (3) simultaneous variations of both the momentum and energy sources.

\begin{figure}[H]
\begin{center}
\includegraphics[width=0.6\textwidth]{ 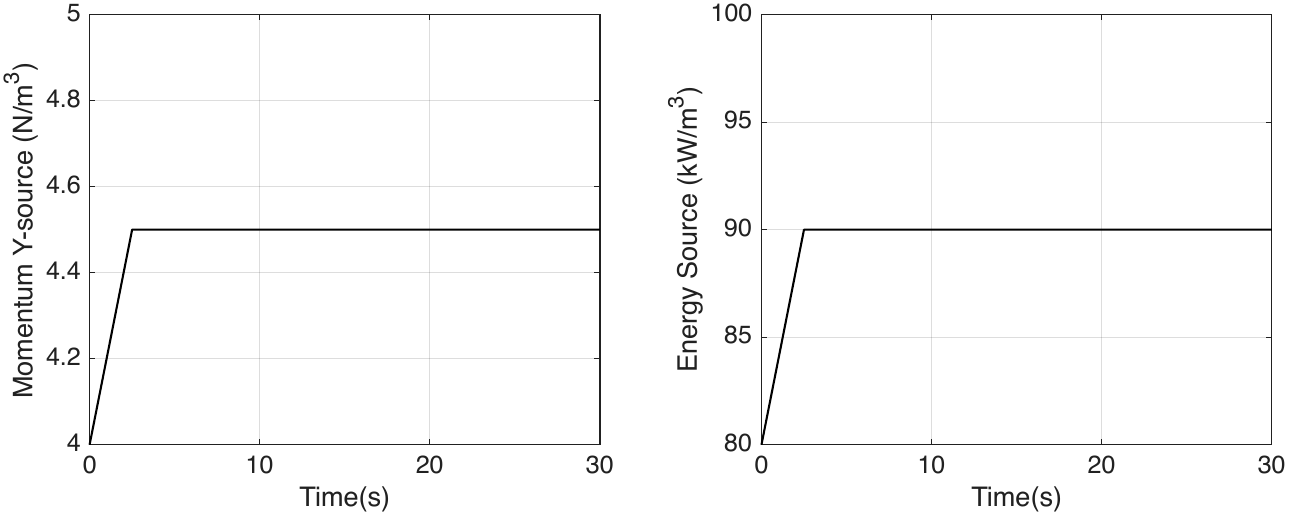}
\caption{Momentum and energy source variations applied to the systems.}
\label{fig:E2f2}
\end{center}
\end{figure}

Figure \ref{fig:E2f3} shows the relative $L^2$ error as a function of time of the three cases, which is defined as the following: 
\begin{gather}
	e_{\vecf v}\left(t\right)=\frac{\left\|\hat{\vecf v}_\delta\left(t\right)-\vecf v_\delta\left(t\right)\right\|_{L^2}}{\left\|\vecf v_\delta\left(t\right)\right\|_{L^2}},\quad
	e_{T}(t)=\frac{\bigl\|\hat{T}_\delta\left(t\right)-T_\delta\left(t\right)\bigr\|_{L^2}}{\left\|T_\delta\left(t\right)\right\|_{L^2}}
\end{gather}
where $\hat{\vecf v}_\delta$ and $\hat{T}_\delta$ denote $\vecf v_\delta$ and $T_\delta$ predicted by the ROM, respectively. The velocity relative $L^2$ errors remain small in all three cases, although the error is slightly higher in case 2 because the velocity variation, induced solely by the temperature change, is much smaller in magnitude, making the relative error more sensitive to small absolute differences. For temperature relative $L^2$ error, ROM gives a good result on all three cases, where the relative $L^2$ error decreases rapidly to below $1\%$ fast and drops to less than $0.5\%$ as the solution approaches the steady state. 

\begin{figure}[H]
    \centering
    \begin{subfigure}[t]{0.36\textwidth}
        \centering
        \includegraphics[width=\linewidth]{ 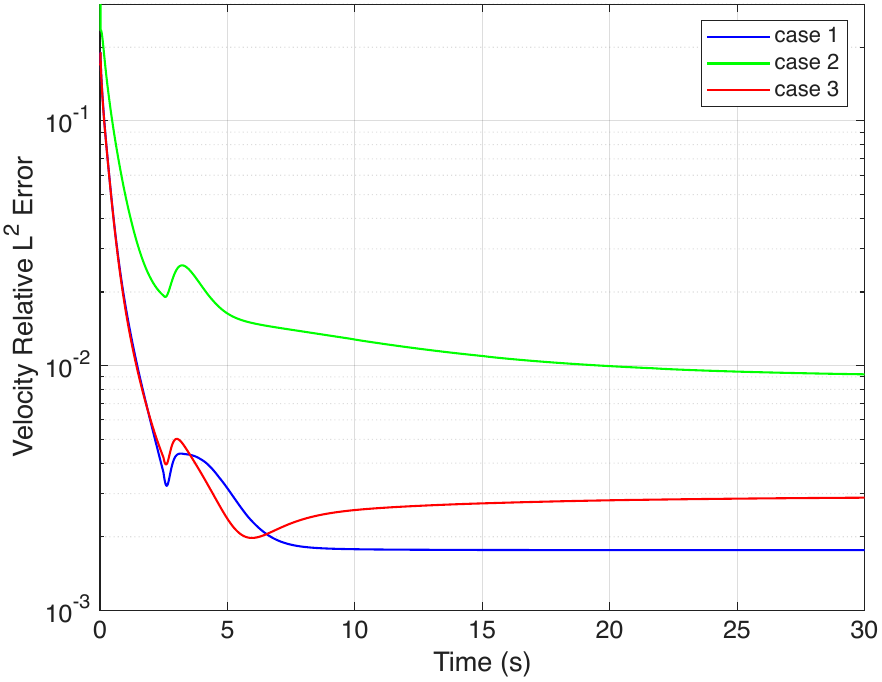}
        %\caption{Velocity relative $L^2$ error}
        \label{fig:E2f3_1}
    \end{subfigure}
    \hspace{0.02\textwidth}
    \begin{subfigure}[t]{0.36\textwidth}
        \centering
        \includegraphics[width=\linewidth]{ 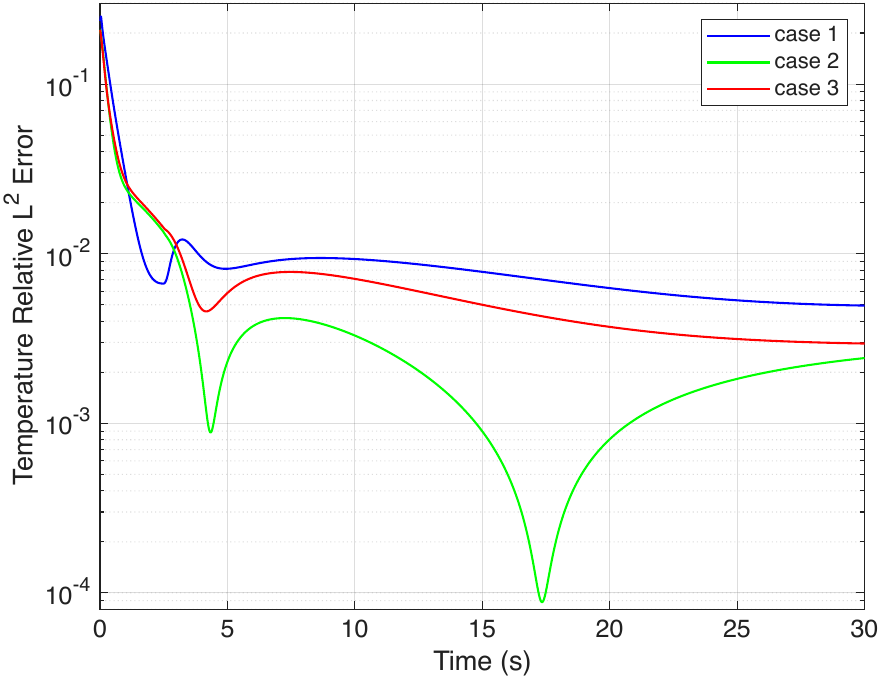}
        %\caption{Temperature relative $L^2$ error}
        \label{fig:E2f3_2}
    \end{subfigure}
    \caption{Relative $L^2$ error as a function of time for three cases: (1) variation of the momentum source, $\vecf{f}_\delta$, only; (2) variation of the energy source, $q_\delta$, only; and (3) simultaneous variations of both the momentum and energy sources.}
    \label{fig:E2f3}
\end{figure}

The transient response of $\hat{\vecf v}_\delta$ and $\hat{T}_\delta$ at the six monitoring points labeled $a$ to $f$ in Figure \ref{fig:E2f1} are compared with the corresponding FOM solutions, $\vecf v_\delta$ and $T_\delta$, for the three test cases, as shown in Figure \ref{fig:E2f4}. The ROM accurately predicts the transient responses at all monitoring points, with the ROM and FOM results almost indistinguishable throughout the simulations. 

Figure \ref{fig:E2f5} further compares the final steady-state distributions of $\hat{\vecf v}_\delta$ and $\hat{T}_\delta$ with those of $\vecf v_\delta$ and $T_\delta$, together with the corresponding absolute error fields, for the case with simultaneous variations of both the momentum and energy sources. The ROM predictions agree well with the FOM results from the figure.

\begin{figure}[H]
\begin{center}
\includegraphics[width=0.95\textwidth]{ 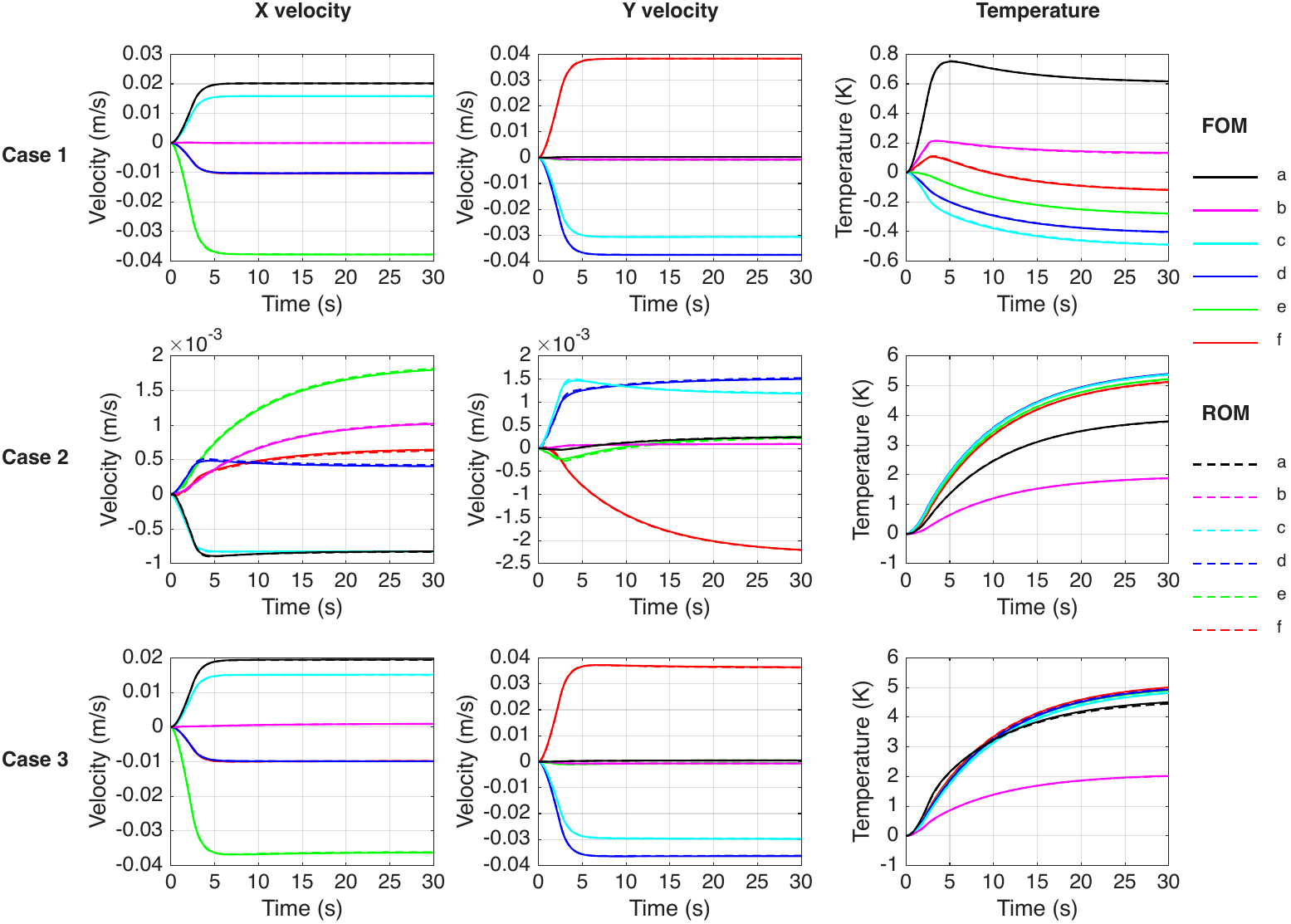}
\caption{Transient responses of $\hat{\vecf v}_\delta$ and $\hat{T}_\delta$ at the six monitoring points compared with the corresponding FOM solutions, $\vecf v_\delta$ and $T_\delta$, for three test cases: (1) variation of the momentum source, $\vecf{f}_\delta$, only; (2) variation of the energy source, $q_\delta$, only; and (3) simultaneous variations of both the momentum and energy sources.}
\label{fig:E2f4}
\end{center}
\end{figure}

\begin{figure}[H]
\begin{center}
\includegraphics[width=0.95\textwidth]{ 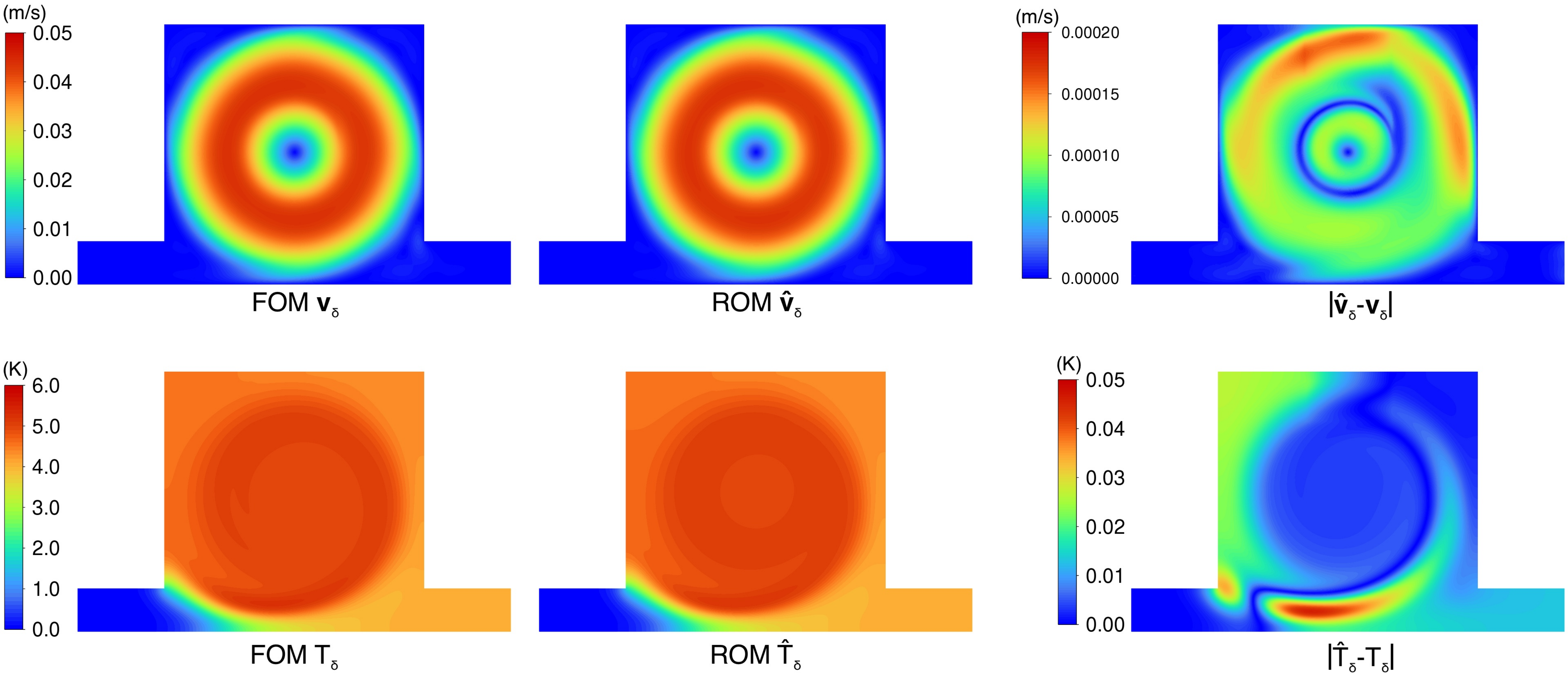}
\caption{Comparison of the steady-state distributions of $\hat{\vecf v}_\delta$ and $\hat{T}_\delta$ with the corresponding FOM solutions, $\vecf v_\delta$ and $T_\delta$, together with the corresponding absolute error fields, for the case with simultaneous variations of both the momentum and energy sources.}
\label{fig:E2f5}
\end{center}
\end{figure}

\section{Conclusion}

This paper presents a novel moment-matching method for building ROMs of multi-input quadratic-bilinear descriptor systems. The proposed method of deriving a set of compact models allows building a numerically stable ROM for a QB system while matching the same transfer functions as the traditional Krylov-based construction based on a single large model. 
The procedure of ROM derivation is completely matrix-free, avoiding the need for explicit construction and storage of the full-order system matrices and tensors. This feature is particularly important for large-scale CFD problems, for which explicit construction and storage of the full-order matrices and tensors is intractable. The ROM in this study is successfully implemented and evaluated for QB systems with multiple inputs, including those acting on the algebraic equations. The results of numerical experiments conducted for multiple time-varying inputs demonstrate that the transient responses of the ROMs agree well with those of the FOMs. 
It should be noted that, despite the relatively small scales of two-dimensional examples used in this study, the proposed method will be equally applicable to large-scale three-dimensional models. Future work will focus on applying the method to three-dimensional turbulent flow systems and evaluating their performance.

\section*{Acknowledgments}
This work is supported by the Cyber-Physical Systems program, project award No.~2024-67021-41626, from the U.S. Department of Agriculture's National Institute of Food and Agriculture.

\appendix
\appendixpage
\section{Transformation of a General QB Descriptor System}\label{DS}

\renewcommand{\theequation}{\thesection.\arabic{equation}}
\setcounter{equation}{0}

Consider the rank normal form of $\vect A_{22}$ in Equation (\ref{QBeq}):
\begin{gather}
\vect A_{22}=\vect P
\begin{bmatrix}
	\vect I_{r}	& \vect 0\\
	\vect 0		& \vect 0
\end{bmatrix}\vect Q
\end{gather}
Where $\vect P$ and $\vect Q$ are non-singular matrices. Let $\vect y=[\vect y_2,\vect y_3]^T=\vect Q\vect x_2$, where $\vect y_2$ has a dimension $r$, , corresponding to the identity block. The algebraic equation becomes:
\begin{gather}
\vect 0=\vect P^{-1}\vect A_{21}\vect x_1+\begin{bmatrix}
	\vect I_{r}	&\vect 0\\
	\vect 0		&\vect 0
\end{bmatrix}
\begin{bmatrix}
	\vect y_2\\
	\vect y_3
\end{bmatrix}+\vect P^{-1}\vect B_2\vect u
\end{gather}
Let $\hat{\vect A}_{21}$ denote the first $r$ rows of $\vect P^{-1}\vect A_{21}$, and let $\hat{\vect A}_{31}$ denote the remaining rows, and let $\hat {\vect B}_{2}$ denote the first $r$ rows of $\vect P^{-1}{\vect B}_2$, and let $\hat {\vect B}_{3}$ denote the remaining rows:
\begin{subequations}
\begin{gather}
\vect 0=\hat {\vect A}_{21}\vect x_1+\vect y_2+\hat {\vect B}_{2}\vect u	 \label{AEinv}\\
\vect 0=\hat {\vect A}_{31}\vect x_1+\hat {\vect B}_{3}\vect u	
\end{gather}
\end{subequations}
Then, the $\vect x_2$ can be expressed as, using Equation (\ref{AEinv}):
\begin{gather}
	\vect x_2=\vect Q^{-1}
	\begin{bmatrix}
	-\hat {\vect A}_{21}\vect x_1-\hat {\vect B}_{2}\vect u	\\
	\vect y_3		
\end{bmatrix}
\end{gather}
Denote the first $r$ column of the $\vect A_{12}\vect Q^{-1}$ as $\hat {\vect A}_{12}$, and the rest parts as $\hat {\vect A}_{13}$, and define the following operators:
\begin{subequations}
\begin{gather}
\hat {\vect A}_{11}=\vect A_{11}-\hat {\vect A}_{12}\hat {\vect A}_{21},\quad
\hat {\vect B}_{1}=\vect B_1-\hat{\vect A}_{12}\hat{\vect B}_{2},\\
\hat {\tensor H}_{11}\left(\vect x_1, \vect x_1\right)=
	\tensor H_{11}\left(\vect x_1, \vect x_1\right)
	+\tensor H_{12}\left(\vect x_1, \vect Q^{-1}
	\begin{bmatrix}
	-\hat {\vect A}_{21}\vect x_1	\\
	\vect 0		
	\end{bmatrix}\right),\\
\hat {\tensor H}_{13}\left(\vect x_1, \vect y_3\right)=
	\tensor H_{12}\left(\vect x_1, \vect Q^{-1}
	\begin{bmatrix}
	\vect 0	\\
	\vect y_3		
	\end{bmatrix}\right),\\
\hat {\tensor N}_1\left(\vect x_1\right)\vect u=
	\tensor N_1\left(\vect x_1\right)\vect u+
	\tensor H_{12}\left(\vect x_1, \vect Q^{-1}
	\begin{bmatrix}
	-\hat {\vect B}_{2}\vect u	\\
	\vect 0		
	\end{bmatrix}\right)+\tensor N_2\left(\vect Q^{-1}
	\begin{bmatrix}
	-\hat {\vect A}_{21}\vect x_1	\\
	\vect 0		
	\end{bmatrix}
	\right)\vect u,\end{gather}\begin{gather}
\hat {\tensor N}_3\left(\vect y_3\right)\vect u=
	\tensor N_2\left(\vect Q^{-1}
	\begin{bmatrix}
	\vect 0	\\
	\vect y_3		
	\end{bmatrix}
	\right)\vect u,\quad
\hat {\tensor D}_1\left(\vect u, \vect u\right)	=
\tensor N_2\left(\vect Q^{-1}
	\begin{bmatrix}
	\vect -\hat {\vect B}_{2}\vect u	\\
	\vect 0		
	\end{bmatrix}
	\right)\vect u
\end{gather}
\end{subequations}
Then, the quadratic-bilinear descriptor system becomes the following DAEs:
\begin{subequations}
\begin{gather}
 \dot{\vect x_1}=\hat {\vect A}_{11}\vect x_1
  +\hat{\vect A}_{13}\vect y_3+\hat {\vect B}_{1}\vect u
  +\hat {\tensor H}_{11}\left(\vect x_1, \vect x_1\right)
  +\hat {\tensor H}_{13}\left(\vect x_1, \vect y_3\right)
  +\hat {\tensor N}_1\left(\vect x_1\right)\vect u
  +\hat {\tensor N}_3\left(\vect y_3\right)\vect u
  +\hat {\tensor D}_1\left(\vect u, \vect u\right)\\
 \vect 0=\hat {\vect A}_{31}\vect x_1 +\hat{\vect B}_3\vect u 
 \end{gather}
\end{subequations}

\bibliographystyle{IEEEtranN}
\bibliography{Reference.bib}
\end{document}